\documentclass[10pt,oneside]{amsart}
\usepackage{geometry}
\usepackage{
 amsmath, 
 amsxtra, 
 amsthm, 
 amssymb, 
 etex, 
 mathrsfs, 
 mathtools, 
 tikz-cd, 
 bbm,
 xr,
 comment,
 enumitem}
\usepackage{cancel}
\usepackage[all]{xy}
\usepackage{hyperref}
\usepackage{url}
\usepackage{tipa}
\usepackage[normalem]{ulem}
\usepackage{dsfont}
\usepackage{colonequals}
\usepackage{stmaryrd}
\usepackage{color}
\usepackage{upgreek}
\usepackage[normalem]{ulem}

\newtheorem{theorem}{Theorem}[section]
\newtheorem{lemma}[theorem]{Lemma}

\newtheorem{proposition}[theorem]{Proposition}

\newtheorem{defn}[theorem]{Definition}

\numberwithin{equation}{section}
\newtheorem{lthm}{Theorem} % theorems with letters (for intro)

\theoremstyle{remark}
\newtheorem{remark}[theorem]{Remark}

\usepackage[utf8]{inputenc}

\newcommand\EatDot[1]{}

\newcommand{\cyc}{{\mathrm{cyc}}}

\newcommand{\cF}{{\mathcal{F}}}

\newcommand{\gq}{{\mathfrak{q}}}

\newcommand{\gp}{{\mathfrak{p}}}

\newcommand{\Yprime}{\mathcal{Y}^\prime}
\newcommand{\Hf}{H^1_{/f}}
\newcommand{\Span}{\mathrm{Span}}

\newcommand{\Fil}{\mathrm{Fil}}

\newcommand{\Exp}{{\mathrm{exp}}}
\newcommand{\Gal}{{\mathrm{Gal}}}
\newcommand{\Ker}{{\mathrm{Ker}}}
\newcommand{\Coker}{{\mathrm{Coker}}}
\newcommand{\Sel}{{\mathrm{Sel}}}
\newcommand{\GL}{{\mathrm{GL}}}

\newcommand{\cris}{{\mathrm{cris}}}

\newcommand{\Hom}{\mathrm{Hom}}

\newcommand{\loc}{\mathrm{loc}}
\newcommand{\Ind}{\mathrm{Ind}}

\newcommand{\Q}{{\mathbb Q}}
\newcommand{\Z}{{\mathbb Z}}

\newcommand{\A}{{\mathbb A}}

\newcommand{\C}{\mathbb{C}}

\newcommand{\Dcris}{\mathbb{D}_{\mathrm{cris}}}

\DeclareMathOperator{\val}{val}
\DeclareMathOperator{\coker}{coker}

\newcommand{\frob}{\varphi}

\newcommand{\col}{\mathrm{Col}}
\newcommand{\QQ}{\mathbb{Q}}
\newcommand{\ZZ}{\mathbb{Z}}
\newcommand{\Qp}{\mathbb{Q}_p}
\newcommand{\Zp}{\mathbb{Z}_p}

\newcommand{\rank}{\mathrm{rank}}
\newcommand{\HIw}{H^1_{\mathrm{Iw}}}
\definecolor{Green}{rgb}{0.0, 0.5, 0.0}

\title[Mordell--Weil ranks of abelian varieties]{On the Mordell--Weil ranks of supersingular abelian varieties  over $\Z_p^2$-extensions}

\author[C. Dion]{Cédric Dion}
\address[Dion]{D\'epartement de Math\'ematiques et de Statistique\\
Universit\'e Laval, Pavillion Alexandre-Vachon\\
1045 Avenue de la M\'edecine\\
Qu\'ebec, QC\\
Canada G1V 0A6}
\email{cedric.dion.1@ulaval.ca}
\author[J. Ray]{Jishnu Ray}
\address[Ray]{Harish Chandra Research Institute, A CI of Homi Bhabha National Institute,  Chhatnag Road, Jhunsi, Prayagraj (Allahabad) 211 019 India}

\email{jishnuray@hri.res.in; jishnuray1992@gmail.com}



\begin{document}

\maketitle
\setcounter{tocdepth}{1}
%\tableofcontents
  
\begin{abstract}
%Let $A$ be an abelian variety \blue{Is $A$ defined over $\QQ$?} with supersingular reduction at a fixed odd prime $p$ and $K$ an imaginary quadratic field in which $p$ splits. Under certain assumptions, we give a growth estimate for the Mordell-Weil rank of $A$ over finite extensions inside the $\Z_p^2$-extension of $K$. \\

Let $p$ be a fixed odd prime and let $K$ be an imaginary quadratic field in which $p$ splits. Let $A$ be an abelian variety defined over $K$ with supersingular reduction at both primes above $p$ in $K$. Under certain assumptions, we give a growth estimate for the Mordell--Weil rank of $A$ over finite extensions inside the $\Z_p^2$-extension of $K$. In the last section, written by Chris Williams, he includes some speculative remarks on the $p$-adic $L$-functions  for $\mathrm{GSp}(4)$ corresponding to the multi-signed Selmer groups constructed in this paper.
\end{abstract}
\tableofcontents
\section{Introduction}
\subsection{Work of Lei--Ponsinet}Let $p$ be an odd prime, $L$ be a number field and $L_\cyc$ be the cyclotomic $\Z_p$-extension of $L$. Let $L_\cyc=\cup_n L_n$ where $\Gal(L_n/L)\cong \Z/p^n\Z$. Let $A$ be an abelian variety over $L$. Suppose $A$ has good ordinary reduction at the primes of $L$ above $p$. Then it is conjectured by Mazur \cite{Mazur} that the classical $p$-primary Selmer group $\Sel_p(A/L_\cyc)$ is cotorsion as a $\Z_p[[\Gal(L_\cyc/L)]]$-module (follows by the work of Kato \cite{kato} and Rohrlich \cite{Roh84} for  elliptic curves defined over $\Q$ and $L/\Q$ abelian). Under this assumption Mazur's control theorem \cite{Mazur} on these classical $p$-primary Selmer groups implies that the Mordell--Weil rank $A(L_n)$ is bounded independently of $n$. However, this technique only works under the ordinary assumption. If the abelian variety has supersingular reduction at one of the primes above $p$, then the $\Sel_p(A/L_\cyc)$ is not cotorsion and Mazur's control theorem will not hold. Now suppose that $p$ is unramified in $L$  and $A$ has supersingular reduction at all primes of $L$ above $p$. Then Büyükboduk and Lei \cite{buyukboduklei0} constructed signed Selmer groups for $A$; these generalize plus and minus Selmer groups of Kobayashi \cite{kobayashi03} for $a_p=0$ and Sprung's $\sharp / \flat$-Selmer groups for $p \mid a_p\neq 0$  \cite{sprungfactorizationinitial}. The former are proved to be cotorsion for elliptic curves over $\Q$ and the latter are conjectured to be cotorsion as well. In the case when $p \mid a_p \neq 0$, at least one of the $\sharp / \flat$-Selmer groups is known to be cotorsion \cite{sprungfactorizationinitial}. In the case of $a_p=0$, Kobayashi also proved an analogue of the control theorem for $\pm$-Selmer groups. Combined, those results lead to a bound independent of $n$ for the growth of the rank for supersingular elliptic curves \cite[Corollary 10.2]{kobayashi03}. Under the assumption that the signed Selmer groups of Büyükboduk and Lei are cotorsion, in \cite{LP20}, Lei and Ponsinet give an explicit sufficient condition in terms of Coleman maps attached to $A$ which ensures that the rank of the Mordell--Weil group of the dual abelian variety $A^\vee(L_n)$ stays bounded as $n$ varies. Further they show that when the Frobenius {operator} on the Dieudonné module at $p$ can be expressed as certain block anti-diagonal matrix (cf. \cite[Section 3.2.1]{LP20}) then their explicit condition is satisfied; hence their result applies to abelian varieties of $GL_2$-type. 
\subsection{Work of Lei--Sprung} %In \cite{LeiSprung2020} (see also \cite{Spr21} for a less technical exposition), Lei and Sprung consider the special case of elliptic curves defined over $\Q$ and the base field  $F=K$, an imaginary {quadratic} field in which $p$ splits. Such a $K$ can have a $\Z_p$-extension where the Mordell--Weil rank can be unbounded. In particular,  the Mordell--Weil rank of {$E$} at the $n$-th layer of a $\Z_p$-extension of $K$ is of the form $ap^n + O(1)$ where $a$ is an integer called the growth number. This growth number $a=0$ for the cyclotomic $\Z_p$ extension of $K$ and \sout{can be equal to $1$ for the anticyclotomic $\Z_p$ extension of $K$} (see for example \cite[Corollary 3]{Bertolini95}). \JR{This is not quite correctly written and I have made bad choice of words.  Also references to Cornut and Vatsal are missing. This is nicely written in the second paragraph of page 184 of \cite{LeiSprung2020}. I guess the words ``give scenarios"  there is crucial. It indicates to the assumption that the root number is negative. See ``Growth number vanishing conjecture 1.2" in page 186 of \cite{LeiSprung2020}. See also Theorem 5.2 of \cite{HL20}. This part needs to be written nicely and concisely following the second paragraph of page 184 of \cite{LeiSprung2020}.} So along the anticyclotomic tower the Mordell--Weil rank can be unbounded. Let $K_\infty$ be the $\Z_p^2$-extension of $K$ obtained as a compositum of the cyclotomic and the anticyclotomic extension of $K$. Let $K_n$ be a subfield of $K_\infty$ such that $\Gal(K_n/K)\cong (\Z/p^n\Z)^2$. Suppose the prime $p$ is supersingular, B.D. Kim constructed four signed Selmer groups for elliptic curves over the $\Z_p^2$-extension $K_\infty$, generalizing works of Kobayashi for the cyclotomic extension. Assume that $E$ is an elliptic curve over $\Q$ with conductor prime to $p$ and the class number of $K$ is coprime to $p$. Under the assumption that the signed Selmer groups of B.D. Kim are cotorsion, the main result shown by Lei and Sprung is that $\rank_{\Z_p} E(K_n) =O(p^n)$ \cite{LeiSprung2020}.
%\JR{Their technique does not need control theorems for signed Selmer groups, since they are unavailable for $p \mid a_p \neq 0$ even for the cyclotomic case.\\}
%\green{Comment: I would replace the previous paragraph by the one in green below. The Work of Lei-Sprung is more about sharp flat Selmer groups than Kim's Selmer groups. I modified the text accordingly.\bigbreak}
{In \cite{LeiSprung2020} (see also \cite{Spr21} for a less technical exposition), Lei and Sprung consider the special case of elliptic curves defined over $\Q$ and the base field  $F=K$, an imaginary {quadratic} field in which $p$ splits. The Mordell--Weil rank of {$E$} at the $n$-th layer of a $\Z_p$-extension of $K$ is know to be of the form $ap^n + O(1)$ where $a$ is an integer called the growth number. When $E$ has ordinary reduction at $p$, this growth number $a$ is equal to $0$ for the cyclotomic $\Z_p$-extension of $K$ by the works of Kato and Rohrlich. But works of \cite{Bertolini95}, \cite{Ne12}, \cite{Ho04}, \cite{Cor02} and \cite{Vat03} exhibit cases where $a=1$ along the anticyclotomic extension of $K$ when the root number of $E/K$ is $-1$ as it is predicted by Mazur's growth number conjecture \cite[\S 18, p.201]{MazGrowthNumber}. In the supersingular $p \geq 5$ case, \cite{longovigni} also finds $a=1$ under some assumptions. So along the anticyclotomic tower, the Mordell--Weil rank may be unbounded. Let $K_\infty$ be the $\Z_p^2$-extension of $K$ obtained as a compositum of the cyclotomic and the anticyclotomic extension of $K$. Let $K_n$ be a subfield of $K_\infty$ such that $\Gal(K_n/K)\cong (\Z/p^n\Z)^2$. Suppose the elliptic curve $E$ has good supersingular reduction at $p$, Lei and Sprung constructed four $\sharp/\flat$-Selmer groups for elliptic curves over the $\Zp^2$-extension $K_\infty$, generalizing works of Kim for the case $a_p=0$ \cite{kimdoublysigned}. Assume that $E$ is an elliptic curve over $\Q$ with conductor prime to $p$ and the class number of $K$ is coprime to $p$. When $p|a_p$, but $a_p\neq 0$, suppose that one of the four $\pm\pm$-Selmer groups defined by Kim \cite{kimdoublysigned} is cotorsion. Then, Lei and Sprung show that $\rank_{} E(K_n) =O(p^n)$. In the case $a_p=0$, they show that $\rank_{} E(K_n) =O(p^n)$ under the assumption that all of the four $\sharp/\flat\sharp/\flat$-Selmer groups they defined are cotorsion.
\subsection{Our work in this article}
{ Kazim Büyükboduk  and Antonio Lei jointly started a program to generalize the plus/minus theory of  Kobayashi and Pollack in the case of motives crystalline at $p$. This led to their series of joint papers \cite{buyukboduklei0}, \cite{buyukbodukleiPLMS} and \cite{BL21}. Note that \cite{buyukboduklei0} and \cite{buyukbodukleiPLMS} (including the work of Lei--Ponsinet mentioned above) are over the cyclotomic $\Z_p$-extension and \cite{BL21} is for representations coming from modular forms. In this paper our results add to that program giving us results for abelian varieties over a maximal abelian pro-$p$ extension of a number field that is unramified outside $p$.}

We restrict ourselves to the case of imaginary quadratic field $K$ where $p$ splits. But instead of working with elliptic curves we work with abelian varieties with supersingular reduction at both primes above $p$. In this case, we have at our disposal the signed Selmer groups of Büyükboduk and Lei \cite{buyukboduklei0} which are defined only over the cyclotomic extension of $K$. The objectives in our paper are threefold.
We 
\begin{itemize}
    \item define {multi-}signed Selmer groups for $A$ over the $\Z_p^2$-extension $K_\infty$. In order to achieve this we construct multi-signed Coleman maps  attached to the abelian variety $A$. This gives us the local condition at primes above $p$. 
    \item give an explicit sufficient condition in terms of Coleman maps (hypothesis \textbf{(H-large)}) mentioned in the text) attached to $A$ which ensures that the Mordell--Weil rank  is bounded by a function which is $O(p^n)$ along the tower $K_n$  (see theorem \ref{thm:mainone}). %This is done by first estimating the $\Z_p$-ranks of the fine (strict) dual Selmer group along the  tower $K_n$ and then  establishing a proper bound for the $\Z_p$-ranks of the kernel of the quotient map between classical dual Selmer group  and the dual fine Selmer group. 
    \item show that this explicit condition \textbf{(H-large)} is satisfied when the Frobenius on the Dieudonné module at primes above $p$ can be expressed in a certain block diagonal form. This special case  can be thought of as an analogue for the case $a_p=0$ for supersingular elliptic curves and happens for abelian varieties of the $GL_2$-type (see section \ref{sec:blockanti}). When our signed Selmer groups for $A$ are cotorsion, we hence deduce $\rank_{}A^\vee(K_n)=O(p^n)$ (see theorem \ref{thm:Main2}).
\end{itemize}

{In order to state the main results more precisely, let us introduce some more notation first. Let $A^\vee$ be the dual abelian variety and let $\Sel_{\underline{J}}(A^\vee/K_\infty)$ be the multi-signed Selmer group attached to $A$ constructed in section \ref{Selmer groups}. Let $\mathcal{X}_{\underline{J}}(A^\vee/K_\infty)$ denote the Pontryagin dual $$\Hom_{\mathrm{cts}}(\Sel_{\underline{J}}(A^\vee/K_\infty),\Qp/\Zp).$$ Let $\Lambda$ be the completed group ring $\Zp \llbracket \Gal(K_\infty/K) \rrbracket$. Then, the first of our main theorems is the following:
\begin{lthm}[Theorem \ref{thm:mainone}]
Suppose that $\mathcal{X}_{\underline{J}}(A^\vee/K_\infty)$ is a torsion $\Lambda$-module for some $\underline{J}$ and that hypothesis \textbf{(H-large)} is satisfied. We have 
$
\rank_{}A^\vee (K_n) = O(p^n).
$
\end{lthm}
To state our next theorem, we need to introduce some further notation. Let $T_p(A)$ be the Tate module of $A$. Suppose that $p=\gp \gp^c$ where $\gp$ and $\gp^c$ are prime ideals of $K$. Let $\mathbb{D}_{\cris,\gp}(T_p(A))$ (reps. $\mathbb{D}_{\cris,\gp^c}(T_p(A))$) be the Dieudonné module of $T_p(A)$ seen as a representation of $\Gal(\overline{K_\gp}/K_\gp)$ (resp. $\Gal(\overline{K_{\gp^c}}/K_{\gp^c})$). These modules are naturally equipped with an action of a Frobenius operator. Let $C_\gp$ and $C_{\gp^c}$ be the matrices defined in section \ref{Hodge theory} arising from the Frobenius action on $\mathbb{D}_{\cris,\gp}(T_p(A))$ and $\mathbb{D}_{\cris,\gp^c}(T_p(A))$ respectively.
\begin{lthm} [Theorem \ref{thm:Main2}]\label{thmB}
Suppose $C_\gp$ and $C_{\gp^c}$ are block anti-diagonal matrices and the Pontryagin dual of the Selmer groups $\Sel_{\underline{J}}(A^{{\vee}}/K_\infty)$ for $\underline{J}\in \{\underline{I}_0,\underline{I}_1,\underline{I}_{\mathrm{mix}_{0,1}},\underline{I}_{\mathrm{mix}_{1,0}}\}$ are all $\Lambda$-torsion. Then
$
\rank_{}A^\vee (K_n) = O(p^n).
$
\end{lthm}
The Selmer groups appearing in theorem B are all multi-signed Selmer groups for particular choices of the indexing set $\underline{J}$ (see the discussion before proposition \ref{block}). \bigbreak
}
Let the abelian variety $A$ be of dimension $g$ and such that $C_\gp$ and $C_{\gp^c}$ are block anti-diagonal. The indexing set $\underline{J}:=(J_\gp,J_{\gp^c})$ where $J_\gp$ and $J_{\gp^c}$ are subsets of $\{1,...,2g\}$ such that $|J_\gp| + |J_{\gp^c}|=2g$. 
Note that the Selmer groups $\Sel_{\underline{J}}(A^\vee/K_\infty)$ depend not only on the indexing set $\underline{J}$ but also on the choice of Hodge-compatible bases of the Dieudonné modules $\mathbb{D}_{\mathrm{cris},\gp}(T_p(A))$ and $\mathbb{D}_{\mathrm{cris},\gp^c}(T_p(A))$ (see section \ref{Selmer groups} and remark \ref{changebasis}). This is also the case for the multi-signed Selmer groups defined in \cite{buyukboduklei0} and \cite{LP20} for supersingular abelian varieties in the cyclotomic case. When $2g=4$, upon fixing a Hodge-compatible bases, the indexing set $\underline{J}$ has cardinality $70$. But all of these $70$ multi-signed Selmer groups are \textit{not} canonically defined. Defining multi-signed Selmer groups canonically is crucial in framing corresponding Iwasawa main conjectures relating our multi-signed Selmer groups to canonically defined $p$-adic $L$-functions in the analytic side.  A new key observation made in this paper is that the specific \textit{four multi-signed Selmer groups} mentioned in theorem \ref{thmB} are independent of the choice of the Hodge-compatible bases and hence canonically defined (see proposition \ref{canonical}). \footnote{We thank David Loeffler for suggesting this to us.}   Therefore, this gives a strong incentive for regarding these particular multi-signed choices as ``more fundamental and arithmetically interesting" than the other multi-signed Selmer groups. (In the case of elliptic curves over the $\Z_p^2$-extension $K_\infty$, they correspond to the usual four  $\sharp/\flat$ signed Selmer groups of Lei and Sprung, generalizing four $+/-$ signed Selmer groups of Kim when $a_p=0$.) Apart from these four multi-signed Selmer groups, we also show that the \textit{two other multi-signed Selmer groups} corresponding to the choices $\underline{J}=(\{1,...,2g\}, \emptyset)$ and $\underline{J}=(\emptyset, \{1,...,2g\})$ are also canonically defined and independent of the choice of Hodge-compatible bases (see proposition \ref{canonical}). These two Selmer groups should correspond to BDP type $p$-adic $L$-functions (see section \ref{sec:speculative remarks}).

We now give an outline of the paper. In section \ref{sec:prem}, we introduce our notations and recall the basic preliminaries on $p$-adic Hodge theory that we will need throughout our article.  The definition of multi-signed Selmer groups and Coleman maps are given in section \ref{sec: multi}. {In section 4, under the hypothesis \textbf{(H-large)}, we use multi-signed Coleman maps and logarithmic matrices to study the growth of
$$
\mathcal{Y}_n^\prime \colonequals \Coker\left( H_{\mathrm{Iw}}^1(K_\infty,T)_{\Gamma_n}\to \prod_{v|p}\frac{H^1(K_{n,v},T)}{H^1_f(K_{n,v},T)} \right)
$$
where $H^1_f(K_{n,v},T)$ is the Bloch-Kato local condition defined using the kernel of the dual exponential map. To achieve this, we generalized the work of \cite{LP20} to the setting of $\mathbb{Z}_p^2$-extensions by using inputs from \cite{LeiSprung2020}. We find that $\rank_{\Zp}(\mathcal{Y}_n^\prime)=O(p^n)$. In section 5, we consider the short exact sequence
$$
0\to \mathcal{Y}_n \to \mathcal{X}_n \to \mathcal{X}_n^0 \to 0
$$
where $\mathcal{X}_n$ is the Pontryagin dual of the classical $p$-Selmer group of $A^\vee$ over $K_n$, $\mathcal{X}_n^0$ is the Pontryagin dual of the fine Selmer group of $A^\vee$ and
$$
\mathcal{Y}_n \colonequals \Coker \left( H^1_\Sigma(K_n,T) \to \prod_{v|p}\frac{H^1(K_{n,v},T)}{H^1_f(K_{n,v},T)} \right).
$$
Under the assumption that at least one of the Selmer groups $\Sel_{\underline{J}}(A^\vee/K_\infty)$ is $\Lambda$-cotorsion, we show that $\rank_{\Zp}(\mathcal{X}_n^0)$ is also $O(p^n)$. Using the natural surjection $\mathcal{Y}_n^\prime\to \mathcal{Y}_n$, we deduce that the growth of $\mathcal{Y}_n$ is at most $O(p^n)$. We are then able to conclude that the Mordell-Weil rank of $A^\vee$ is $O(p^n)$ along $K_\infty$.}

{One expects that there should be analytic analogues of all of the above, corresponding to our (algebraic) results under appropriate Iwasawa main conjectures. In section \ref{sec:speculative remarks}, written by Chris Williams, we give some speculative remarks on the shape of the analytic theory.}

The above concludes the outline. Now, before going to the next section, we would like to make further remarks. In this paper we stick to the case of  $\Z_p^2$-extension of an imaginary quadratic field instead of working more generally over $\Z_p^d$-extension of a number field. Our guess is that the growth condition will be $O(p^{(d-1)n})$, but we will need to make the hypotheses that the multi-signed Selmer groups are cotorsion over the corresponding Iwasawa algebra. Such hypotheses are  currently only known for elliptic curves under non-vanishing of signed $p$-adic $L$-functions over $\Z_p^2$-extension of an imaginary quadratic field (see remark \ref{rem:4.3}).
The analytic side for general abelian varieties are even more mysterious for $d>2$ (see section \ref{sec:speculative remarks}), and in this case (the Euler system machinery is also unavailable) as per our knowledge,  there are currently no results in the literature towards cotorsion-ness of signed Selmer groups.
Even if we assume that the signed Selmer groups are cotorsion, the present techniques in this paper will not generalize verbatim for $d>2$. We also faced technical difficulty generalizing \eqref{tech} for general $\Z_p^d$-extensions.

\section*{Acknowledgements}
We are grateful to Antonio Lei for answering many of our questions. We thank David Loeffler and Chris Williams for several useful conversations. {We also thank Kâz\i m Büyükboduk, Eknath Ghate,  Mahesh Kakde, Chan-Ho Kim, Debanjana Kundu, Meng Fai Lim,  Filippo  Nuccio, Gautier Ponsinet, Dipendra Prasad, Anwesh Ray and  Romyar Sharifi  for comments and corrections that helped improving the quality of the paper.} Finally, we are grateful for the referee's many constructive comments. The first named author's research is supported by the Canada Graduate Scholarships – Doctoral program from the Natural Sciences and Engineering Research Council of Canada. The second author gratefully acknowledges support from the Inspire Research Grant, DST, Govt. of India. 

%  Our technique of the proof generalizes the techniques used by Lei-Ponsinet \cite{LP20} and Lei-Sprung \cite{LeiSprung2020}. 

\section{Preliminaries}\label{sec:prem}
\subsection{Local and Global setup}
Let $p \geq 3$ be a fixed prime number for the rest of the paper. Let $K$ be an imaginary quadratic field in which $p$ splits into the primes $\gp$ and $\gp^c$ of $K$. Here, $c$ denotes the complex conjugation. We will always use $\gq$ to mean an element of the set $\{\gp,\gp^c\}$. Let $K_\infty$ be the unique $\Zp^2$-extension of $K$ with Galois group $\Gamma \colonequals \Gal(K_\infty/K)$. If $\mathfrak{a}$ is an ideal of $\mathcal{O}_K$, $K(\mathfrak{a})$ will denote the ray class field of $K$ of conductor $\mathfrak{a}$. If $n \geq 0$ is an integer, write $\Gamma_n \colonequals \Gamma^{p^n}$ and $K_n \colonequals K_\infty^{\Gamma_n} = K(p^{n+1})^{\Gal(K(1)/K)}$. We make the hypothesis that $K(1)\bigcap K_\infty=K$. It follows that the Galois group $\Gamma$ is isomorphic to $G_{\gp}\times G_{\gp^c}$ where $G_\gq$ is the Galois group of the extension $K(\gq^\infty) \bigcap K_\infty / K$. We fix topological generators $\gamma_\gp$ and $\gamma_{\gp^c}$ respectively for these groups. Let $\Lambda \colonequals \Zp \llbracket \Gamma \rrbracket \cong \Zp \llbracket \gamma_\gp-1, \gamma_{\gp^c}-1\rrbracket$. Sometimes, we shall also use the notation $\Lambda(\Gamma)$ for $\Lambda$ when we want to emphasize that $\Lambda$ is the Iwasawa algebra of the group $\Gamma$. More generally, if $G$ is any profinite group, we denote by $\Lambda(G)$ the completed group ring $\mathbb{Z}_p \llbracket G \rrbracket$. Write $\Gamma_\gq$ for the decomposition group of $\gq$ in $\Gamma$.  Let $\mu_{p^n}$ denote the set of $p^n$th roots of unity and let
$\mu_{p^\infty}\colonequals \bigcup_{n\geq 1}\mu_{p^n}$.

We let $F_\infty$ and $k_\text{cyc}$ be the unramified $\Zp$-extension and the cyclotomic $\Zp$-extension of $\Qp$ respectively. Let $k_\infty$ be the compositum of $F_\infty$ and $k_\text{cyc}$. For $n \geq 0$, $k_n$ and $F_n$ denote the subextension of $k_\text{cyc}$ and $F_\infty$ such that $[k_n:\Qp]=p^n$ and $[F_n:\Qp]=p^n$. Write $\Gamma_p \colonequals \Gal(k_\infty/\Qp) \cong \Zp^2$, $\Gamma_\text{ur} \colonequals \Gal(F_\infty/\Qp) = \overline{\langle \sigma \rangle} \cong \Zp$ and $\Gamma_\text{cyc} \colonequals \Gal(k_\text{cyc}/\Qp) = \overline{\langle \gamma\rangle} \cong \Zp$. We identify $K_\gq$ with $\Qp$, $\Gamma_\gq$ with $\Gamma_p$ and $G_\gq$ with $\Gamma_\text{cyc}$ via $\gamma_\gq \mapsto \gamma$. Let $\mathcal{H}(\Gamma_\text{cyc})$ be the set of power series
$$
\sum_{n\geq 0} c_{n} \cdot (\gamma -1 )^n
$$
with coefficients in $\mathbf{Q}_p$ such that $\sum_{n\geq 0}c_{n}X^n$ converges on the open unit disk.

\subsection{$p$-adic Hodge Theory}\label{Hodge theory}
For this subsection only, let $K$ be any number field where all primes above $p$ are unramified. Note that for our purpose, $p$ will be totally split in $K$ and so we can take $K_v= \Qp$ in the following discussion {where $v$ is a prime of $K$ above $p$}. Denote by $\mathcal{O}_{K_v}$ the ring of integers of $K_v$ and let $\chi:\Gal(K_v(\mu_{p^\infty})/K_v) \to \Zp^\times$ be the cyclotomic character. Let $\mathcal{M}_{/K}$ be a motive defined over $K$ {in the sense of \cite{FPR94}} and $\mathcal{M}_p$ its $p$-adic realization. Let $T$ be a $G_K$-stable $\Zp$-lattice inside $\mathcal{M}_p$. We shall denote by $T^\dagger \colonequals \Hom(T,\mu_{p^\infty})$ the cartier dual of $T$ and by $T^\ast(1) \colonequals \Hom(T,\Zp(1))$ the Tate dual of $T$. Suppose that \medbreak
\noindent \textbf{(H.crys)} $\mathcal{M}_p$ is crystalline at all the primes $v$ above $p$ in $K$.\medbreak
For simplicity, suppose that the dimension of $\Ind_{K/\QQ}\mathcal{M}_p$ over $K_v$ is even. This will be the case for example when the motive $\mathcal{M}$ is the motive associated to an abelian variety. Let $2g\colonequals \mathrm{dim}_{K_v}(\Ind_{K/\QQ}\mathcal{M}_p)$ and let $g_\pm := \mathrm{dim}_{K_v}(\Ind_{K/\QQ}\mathcal{M}_p)^{c=\pm 1}$. Let $v$ be a prime above $p$ in $K$ and let $\mathbb{D}_{\cris,v}(\mathcal{M}_p)$ be $(\mathbb{B}_\cris\otimes_{\Qp} \mathcal{M}_p)^{G_{K_v}}$ where $\mathbb{B}_\cris$ is the crystalline period ring defined by Fontaine. It admits the structure of a filtered $\frob$-module. Let $\mathbb{A}_{K_v}^+ \colonequals \mathcal{O}_{K_v} \llbracket \pi \rrbracket$ where $\pi$ is seen as a formal variable equipped with a semilinear action by a Frobenius $\varphi$ which acts as the absolute Frobenius on $\mathcal{O}_{K_v}$ and on $\pi$ by $\varphi(\pi) = (\pi+1)^p-1$, and with an action of $g\in \Gal(K_v(\mu_{p^\infty})/K_v)$ given by $g(\pi) = (\pi+1)^{\chi(g)}-1$.
%The ring $\mathbb{A}_{K_v}^+$ is equipped with a Frobenius action defined as the arithmetic Frobenius on $\mathcal{O}_{K_v}$ and as $\frob: \pi \to (1+\pi)^p-1$ on $\pi$ and a $\Gamma_\text{cyc}$-action $\sigma^\prime: \pi \to (1+\pi)^{\chi(\sigma^\prime)}-1$. Let $\mathbb{A}_{K_v}$ be the $p$-adic completion of $\mathcal{O}_{K_v} \llbracket \pi \rrbracket [\pi^{-1}]$. Write $\mathbb{D}_v(T)$ for the Dieudonné module of $T$ defined by $(\mathbb{A} \otimes_{\mathbf{Z}_p} T)^{H_{K_v}}$ where $\mathbb{A}$ is another period ring [B,p. 1477] and $H_{K_v}$ is the kernel of the cyclotomic character $G_{K_v} \to \mathbf{Z}_p^\times$. The module $\mathbb{D}_v(T)$ is a free $\mathbb{A}_{\Qp}$-module of rank $2g$ equipped with a Frobenius and a residual action of $\Gamma_\text{cyc}$. 
Let $\mathbb{N}_v(T)$ be the Wach module of $T$ whose existence and properties are shown in \cite[Proposition 2.1.1]{berger04}. It is a free $\mathbb{A}_{K_v}^+$-module of rank $2g$. Furthermore, the quotient $\mathbb{N}_v(T)/\pi \mathbb{N}_v(T)$ is identified with a $\mathcal{O}_{K_v}$-lattice of $\mathbb{D}_{\cris,v}(\mathcal{M}_p)$. We denote by $\mathbb{D}_{\cris,v}(T)$ this $\mathcal{O}_{K_v}$-lattice. It is equipped with a filtration of $\mathcal{O}_{K_v}$-modules $\{ \Fil^i \mathbb{D}_{\cris,v}(T) \}_{i \in \ZZ}$ and a Frobenius operator $\frob$. If we suppose that \medbreak
\noindent \textbf{(H.HT)} the Hodge--Tate weights of $\mathcal{M}_p$ are either $0$ or $1$,\medbreak
\noindent the filtration takes the form
$$
\Fil^i \mathbb{D}_{\cris,v}(T) = \begin{cases} 0 & \text{if $i \geq 1$,} \\ \mathbb{D}_{\cris,v}(T) & \text{if $i \leq -1$.} \end{cases}
$$
Note that $\mathbb{D}_{\cris,v}(V) = \mathbb{D}_{\cris,v}(T) \otimes_{\mathcal{O}_{K_v}} K_v$. We also make the following assumptions: \medbreak
\noindent \textbf{(H.Frob)} The slopes of the Frobenius on the Dieudonné module $\mathbb{D}_{\cris,v}(\mathcal{M}_p)$ lie inside $[-1,0)$ and that $1$ is not an eigenvalue; \medbreak
\noindent \textbf{(H.P)} $g_+ = g_-$ ($=g$) and $\mathrm{dim}_{K_v} \Fil^0 \mathbb{D}_{\cris,v}(\mathcal{M}_p)=g$.\medbreak
\begin{defn}\label{Hodge-compatible}
Choose an $\mathcal{O}_{K_v}$-basis $\{v_1,\ldots,v_{2g} \}$ of $\mathbb{D}_{\cris,v}(T)$ such that $\{ v_1,\ldots, v_{g} \}$ is a $\mathcal{O}_{K_v}$-basis of $\Fil^0\mathbb{D}_{\cris,v}(T)$. Such a basis is called Hodge-compatible.
\end{defn}
The matrix of $\frob$ with respect to this basis is of the form
$$
C_{\frob,v} = C_v\left[
\begin{array}{c|c}
I_{g} & 0 \\
\hline
0 & \frac{1}{p} I_{g}
\end{array}
\right]
$$
for some $C_v \in \mathrm{GL}_{2g}(\mathcal{O}_{K_v})$ and where $I_g$ is the identity $g \times g$ matrix. \bigbreak

%There is a natural pairing
%$$
%[\cdot,\cdot]:\mathbb{D}_{\cris,v}(T) \times \mathbb{D}_{\cris,v}(T^\ast(1)) \to \mathbb{D}_{\cris,v}(\Zp(1)) \cong \Zp
%$$
%with respect to which $\Fil^i\mathbb{D}_{\cris,v}(T^\ast(1))$ is the orthogonal complement of $\Fil^{-i}\mathbb{D}_{\cris,v}(T)$ and $\frob^{-1}$ is the dual of $p\frob$.

%\subsection{Signed Selmer groups over the cyclotomic extension}
%\JR{A short exposition of this section is given in Section 1.4 and Section 2.1 of \cite{LP20}. We want to put that part in here.  }
\section{Multi-Signed Selmer groups over  \texorpdfstring{$\Z_p^2$-extension}{Zp2-extension}}\label{sec: multi}

In this section, we define multi-signed Selmer groups for motives $\mathcal{M}_{/K}$ satisfying \textbf{(H.crys)}, \textbf{(H.HT)}, \textbf{(H.Frob)} and \textbf{(H.P)} over $\Z_p^2$-extension of an imaginary quadratic field  generalizing an earlier work of Büyükboduk and Lei \cite{BL21}. \bigbreak

Let $K$ be an imaginary quadratic field where $(p)=\gp \gp^c$ splits. We write $K^\text{cyc}/K$ and $K^\text{ac}/K$ for the cyclotomic and anticyclotomic $\Zp$-extensions contained in $K_\infty$ respectively.

\subsection{Yager module}

Let $L/\Qp$ be a finite unramified extension. For $x \in \mathcal{O}_L$, define 

$$
y_{L/\Qp}(x) \colonequals \sum_{\tau \in \Gal(L/\Qp)} \tau (x) \cdot \tau^{-1} \in \mathcal{O}_L [\Gal(L/\Qp)]. 
$$

Let $S_{L/\Qp}$ be the image of $y_{L/\Qp}$ in $\mathcal{O}_L [\Gal(L/\Qp)]$. In \cite[Section 3.2]{LZ14}, it is shown that there is an isomorphism $y_{F_\infty/\Qp}$ of $\Lambda(\Gamma_\text{ur})$-modules
$$
\varprojlim_{\Qp \subseteq L \subseteq F_\infty} \mathcal{O}_L \cong S_{F_\infty / \Qp}\colonequals \varprojlim_{\Qp \subseteq L \subseteq F_\infty} S_{L/\Qp}
$$
where the inverse limit is taken with respect to the trace maps on the left and the projection maps $\Gal(L^\prime/\Qp) \to \Gal(L/\Qp)$ for $L \subseteq L^\prime$ on the right. By \cite[Proposition 3.2]{LZ14}, $\varprojlim_{\Qp \subseteq L \subseteq F_\infty} \mathcal{O}_L$ is a free $\Lambda(\Gamma_\text{ur})$-module of rank one, thus the Yager module $S_{F_\infty/\Qp}$ is also free of rank one over $\Lambda(\Gamma_\text{ur})$.

\subsection{Interpolation property of Loeffler--Zerbes' big logarithm map }

Let $\{\Omega_{\Qp}\}$ be a basis of the Yager module $S_{F_\infty/\Qp}$. Recall that $T$ is a $G_K$-stable $\Zp$-lattice inside $\mathcal{M}_p$. Let $\mathcal{L}_T^{\infty}$ be the two-variable big logarithm map of Loeffler--Zerbes \cite{LZ14}
$$
\mathcal{L}_T^{\infty}: H^1_\mathrm{Iw}(k_\infty,T) \to \Omega_{\Qp} \cdot \left( \mathcal{H}(\Gamma_\text{cyc})\widehat{\otimes}\Lambda(\Gamma_\text{ur}) \right) \otimes_{\Zp} \Dcris(T).
$$
It is a morphism of $\Lambda(\Gamma_p)$-modules. The completed tensor product $\mathcal{H}(\Gamma_\text{cyc})\widehat{\otimes}\Lambda(\Gamma_\text{ur})$ is isomorphic to $\mathcal{H}(\Gamma_p)$, the set of power series in $\gamma-1$ and $\sigma-1$ with coefficients in $\Qp$ converging on the open unit disk. Thus, one may see the big logarithm map as an application sending elements of the Iwasawa cohomology to two-variables power series tensored with $\Dcris(T)$. For $k_n$ a finite subextension of $k_\infty$, let $\exp_T^\ast : H^1(k_n,T) \to k_n \otimes \mathbb{D}_{\text{cris}}(T)$ be the Bloch--Kato dual exponential map. {We say that the cyclotomic part of a finite order character $\eta:\Gamma_p \to \overline{\Q}_p^\times$ is of conductor $p^{n+1}$ if $\eta$ sends the topological generator $\gamma$ to a primitive $p^n$-th root of unity.} Then, the big logarithm map enjoys the following interpolation property (see \cite[Theorem 4.15]{LZ14}):
\begin{proposition}
Let $x \in H^1_\mathrm{Iw}(k_\infty,T)$. Let $\eta$ be a character on $\Gamma_p$ whose cyclotomic part is of conductor $p^{n+1}>1$. Then we have
$$
\Phi^{-n-1}\mathcal{L}_T^{\infty}(x)(\eta) = \varepsilon(\eta^{-1})\exp_T^\ast(e_\eta x)
$$
where $\varepsilon(\eta^{-1})$ is the $\varepsilon$-factor of $\eta^{-1}$, $e_\eta$ is the idempotent corresponding to $\eta$ and $\Phi$ is the operator which act as the arithmetic Frobenius $\sigma$ on $F_m$ and $\varphi$ on $\Dcris (T)$.
\end{proposition}

For simplicity, let us suppose that the dimension of $\Dcris(T)$ as a $\Zp$-module is even as it will be the case for the applications we have in mind. Write $2g$ for the dimension of $\Dcris(T)$ as a $\Zp$-module where $g$ is some strictly positive integer. Choose $\{v_i\}_{i=1}^{2g}$ a Hodge-compatible $\Zp$-basis of $\Dcris(T)$ as in section \ref{Hodge theory}. 
Let $L$ be a finite unramified extension of $\Qp$ and let $L_\mathrm{cyc}$ denote the cyclotomic $\Zp$-extension of $L$. Then, in \cite{buyukboduklei0}, the authors show the existence of one-variable Coleman maps
$$
\col_{T,L,i}:H^1_\mathrm{Iw}(L_\mathrm{cyc},T) \to \mathcal{O}_L\otimes \Lambda(\Gamma_{\mathrm{cyc}})
$$
for $1 \leq i \leq 2g$. Those maps are compatible with the corestriction maps $H^1_\mathrm{Iw}(F_{m,\mathrm{cyc}},T)\to H^1_\mathrm{Iw}(F_{m-1,\mathrm{cyc}},T)$ and the trace maps $\mathcal{O}_{F_m}\otimes \Lambda(\Gamma_\mathrm{cyc}) \to \mathcal{O}_{F_{m-1}}\otimes \Lambda(\Gamma_\mathrm{cyc})$ (see \cite[Appendix A.1]{LeiZhao} where the Coleman maps $\col_{T,F_m,i}$ are denoted by $\col_{\sharp/\flat,F_m}$ instead). We define two-variable Coleman maps by taking the inverse limit of the $\col_{T,F_m,i}$ as $F_m$ runs through the finite extensions between $F_\infty$ and $\Qp$. In order to get a family of maps landing in $\Lambda(\Gamma_p)$, we further compose it with $y_{F_\infty/\Qp}$. More precisely, the two-variable Coleman maps are defined by
\begin{align*}
\col_{T,i}^\infty:H^1_\mathrm{Iw}(k_\infty,T) &\to \Omega_{\Qp}\cdot \Lambda(\Gamma_p) \\
(z_m) & \mapsto (y_{F_\infty/\Qp}\otimes 1) \circ \varprojlim_{F_m} \col_{T,F_m,i}(z_m).
\end{align*}
where $(z_m) \in \varprojlim_{F_m} H^1_\mathrm{Iw}(F_{m,\mathrm{cyc}},T)$. By identifying $\Omega_{\Qp} \cdot \Lambda(\Gamma_p)$ with $\Lambda(\Gamma_p)$, we omit $\Omega_{\Qp}$ from the notation and see $\col_{T,i}^\infty$ as taking value in $\Lambda(\Gamma_p)$. By combining \cite[Theorem 2.13]{buyukboduklei0} and \cite[Theorem 4.7 (1)]{LZ14}, we see that these Coleman maps decompose the big logarithm map
$$
\mathcal{L}_T^{\infty} = (v_1,\ldots, v_{2g}) \cdot M_{T} \cdot \begin{bmatrix} \col_{T,1}^\infty \\ \vdots \\ \col_{T,2g}^\infty \end{bmatrix}.
$$

The matrix $M_{T}$ is called a logarithmic matrix and is defined in the following way: Let $\Phi_{p^n}(1+X)$ denote the $p^n$th cyclotomic polynomial. For $n \geq 1$, first define the matrices
$$
C_n \colonequals \left[
\begin{array}{c|c}
I_{g} & 0 \\
\hline
0 & \Phi_{p^n}(1+X) I_{g}
\end{array}
\right]
C^{-1}
$$
and $M_n \colonequals (C_\frob)^{n+1}C_n \cdots C_1$ where $C$ and $C_\varphi$ are defined as in section \ref{Hodge theory} with $v=p$. In \cite{buyukboduklei0}, it is shown that the sequence $\{ M_n \}$ converges to a $2g \times 2g$ matrix with entries in $\mathcal{H}(\Gamma_\text{cyc})$ which we call $M_{T}$. Let $\gq \in \{\gp, \gp^c\}$. If we include the primes $\gp$ and $\gp^c$ in our notation, then we have the Coleman maps
\begin{equation}
\col_{T,\gq,i}^\infty: H^1(K_\gq, T \otimes \Lambda(\Gamma_\gq)^\iota) \rightarrow \Z_p[[\Gamma_p]]
\end{equation}
for $i\in \{1,...,2g\}$. They satisfy 
$$
\mathcal{L}_{T,\gq}^{\infty} = (v_1,\ldots, v_{2g}) \cdot M_{T,\gq} \cdot \begin{bmatrix} \col_{T,\gq, 1}^\infty \\ \vdots \\ \col_{T,\gq,2g}^\infty \end{bmatrix}
.$$ For $\gq=\gp$, $M_{T,\gp}=\lim_{r \rightarrow \infty}M_{\gp,r}$. Here
 $M_{\gp,r} \colonequals (C_{\frob,\gp})^{r+1}C_{\gp,r} \cdots C_{\gp,1}$,
$$
C_{\gp,r} \colonequals \left[
\begin{array}{c|c}
I_{g} & 0 \\
\hline
0 & \Phi_{p^r}(\gamma_\gp) I_{g}
\end{array}
\right]
C_\gp^{-1}.
$$
For $\gq =\gp^c$ we have  expressions $M_{T,\gp^c}=\lim_{s \rightarrow \infty}M_{\gp^c,s}$, where $ M_{{\gp^c},s} \colonequals (C_{\frob,\gp^c})^{s+1}C_{\gp^c,s} \cdots C_{{\gp^c},1}$,
$$
C_{\gp^c,s} \colonequals \left[
\begin{array}{c|c}
I_{g} & 0 \\
\hline
0 & \Phi_{p^s}(\gamma_{\gp^c}) I_{g}
\end{array}
\right]
C_{\gp^c}^{-1}.
$$
For any $r\geq 1$, $r$ an integer, define $H_{\gp,r}=C_{\gp,r}\cdots C_{\gp,1}$. Similarly, for any integer $s \geq 1$, we define $H_{\gp^c,s}=C_{\gp^c,s}\cdots C_{\gp^c,1}$. 
\newline Let $H_{r,s}$ be the block diagonal matrix where the blocks on the diagonal are given by $H_{\gp,r}$ and $H_{\gp^c,s}$. Consider the tuples $\underline{I}=(I_\gp,I_{\gp^c})$ and $\underline{J}=(J_\gp,J_{\gp^c})$ where $I_\gq \subseteq \{1,...,2g\}$ and $J_\gq \subseteq \{1,...,2g\}$ for $\gq \in \{\gp,\gp^c\}$. Let us suppose that $|I_\gp| +|I_{\gp^c}|=2g$ and $|J_\gp| +|J_{\gp^c}|=2g$. \newline We define $H_{\underline{I}, \underline{J},r,s}$ to be the $(\underline{I},\underline{J})^{th}$-minor of $H_{r,s}$. {We also define $\underline{I}_0 \colonequals (I_{\gp,0},I_{\gp^c,0})$ where $I_{\gq,0} = \{1,\ldots,g\}$.}

\subsection{Definitions and properties}\label{Selmer groups}
The multi-signed Selmer group can be defined without the assumption that $K(1) \cap K_\infty=K$. For this subsection only, we work in more generality and drop this assumption. Suppose that there are $p^t$ primes above $\gp$ and $\gp^c$ in $K_\infty$. Fix a choice of coset representatives $\gamma_1,\ldots, \gamma_{p^t}$ and $\delta_1,\ldots,\delta_{p^t}$ for $\Gamma/\Gamma_\gp$ and $\Gamma/\Gamma_{\gp^c}$ respectively.
Since $p$ splits in $K$, we can identify $K_\gq$ with $\Qp$ and $\Gamma_\gq$ with $\Gamma_p$. Consider the ``semi-local" decomposition coming from Shapiro's lemma
$$
H^1(K_\gp, T\otimes \Lambda^\iota) = \bigoplus_{j=1}^{p^t} H^1(K_\gp, T \otimes \Lambda(\Gamma_\gp)^\iota)\cdot \gamma_j \cong \bigoplus_{v | \gp}H^1_\text{Iw}(K_{\infty,v},T)
$$
where $v$ runs through the primes above $\gp$ in $K_\infty$. By choosing a Hodge-compatible basis of $\mathbb{D}_{\mathrm{cris},\gp}(T)$, define the Coleman map for $T$ at $\gp$ by
\begin{align*}
\col_{T,\gp,i}^{k_\infty}: H^1(K_\gp,T \otimes \Lambda^\iota) &\to \Lambda(\Gamma) \\
x=\sum_{j=1}^{p^t}x_j \cdot \gamma_j &\mapsto \sum_{j=1}^{p^t} \col_{T,\gp,i}^\infty(x_j)\cdot \gamma_j
\end{align*}
for all $1 \leq i \leq 2g$.
%\blue{Let us define $\mathcal{L}_{T,\gp}^{k_\infty}:=\sum_{j=1}^{p^t}\mathcal{L}_{T,\gp}^{\infty}\cdot \gamma_j$}

Let $\mathcal{L}_{T,\gp}^{k_\infty}=\oplus_{j=1}^{p^t}\mathcal{L}_{T,\gp}^{\infty}\cdot \gamma_j$. Define $\col_{T,\gp^c,i}^{k_\infty}$ and $\mathcal{L}_{T,\gp^c}^{k_\infty}$ in an analogous way. Let $J_\gq$ denote a subset of $\{1,\ldots,2g\}$ and let 
\begin{align*}
\col_{T,J_\gq} : H^1(K_\gq,T \otimes \Lambda^\iota) & \to \bigoplus_{i=1}^{|J_\gq|} \Lambda(\Gamma) \\
z & \mapsto (\col^{k_\infty}_{T,\gq,i}(z))_{i \in J_\gq}.
\end{align*}
Tate's local pairing induces a pairing
$$
\bigoplus_{v|\gq}H^1_\mathrm{Iw}(K_{\infty,v},T) \times \bigoplus_{v|\gq}H^1(K_{\infty,v},T^\dagger) \to \Qp/\Zp
$$
for all places $v$ of $K_\infty$ above $\gq$. We define $H^1_{J_\gq}(K_{\infty,\gq},T^\dagger) \subseteq \bigoplus_{v | \gq}H^1(K_{\infty,v},T^\dagger)$ as the orthogonal complement of $\ker \col_{T,J_\gq}$ under the previous pairing. If $L$ is a finite extension of $K_\gq$, define
$$
H^1_\mathrm{ur}(L,\mathcal{M}_p^\ast(1)) \colonequals \ker \left( H^1(L,\mathcal{M}_p^\ast(1)) \to H^1(L_\mathrm{ur},\mathcal{M}_p^\ast(1)) \right)
$$
where $L_\mathrm{ur}$ is the maximal unramified extension of $L$. Let $H^1_\mathrm{ur}(L,T^\dagger)$ be the image of $H^1_\mathrm{ur}(L,\mathcal{M}_p^\ast(1))$ under the natural map $H^1(L,\mathcal{M}_p^\ast(1)) \to H^1(L, T^\dagger)$. Let
$$
H^1_\mathrm{ur}(K_{\infty,w}, T^\dagger) \colonequals \varinjlim_{\Qp \subseteq L \subseteq K_{\infty,w}} H^1_\mathrm{ur}(L,T^\dagger)
$$
where $w$ is a place in $K_\infty$ and $L$ runs through finite extensions of $\Qp$ contained in $K_{\infty,w}$. Let $\underline{J} \colonequals (J_\gp, J_{\gp^c})$ where as before $J_\gp$ and $J_{\gp^c}$ are subsets of $\{1,\ldots,2g\}$ such that $|J_\gp|+|J_{\gp^c}|=2g$. Fix $\Sigma$ a finite set of prime of $K$ containing the primes above $p$, the archimedean primes and the primes of ramification of $T^\dagger$. Let $K_\Sigma$ be the maximal extension of $K$ unramified outside $\Sigma$. Let $\Sigma^\prime$ be the set of primes of $K_\infty$ lying above the primes in $\Sigma$. If $M$ is any $\Gal(K_\Sigma /K_\infty)$-module, we denote $H^1(K_\Sigma / K_\infty,M)$ by $H^1_\Sigma(K_\infty,M)$.

\begin{defn}\label{Selmer def}
Let
$$
\mathcal{P}_{\Sigma,\underline{J}}(T^\dagger/K_\infty) \colonequals \prod_{w \in \Sigma^\prime, w \nmid p} \frac{H^1(K_{\infty,w},T^\dagger)}{H^1_\mathrm{ur}(K_{\infty,w},T^\dagger)} \times \prod_{\gq | p} \frac{\bigoplus_{v|\gq} H^1(K_{\infty,v},T^\dagger)}{H^1_{J_\gq}(K_{\infty,\gq},T^\dagger)}.
$$
The $\underline{J}$-Selmer group of $T^\dagger$ over $K_\infty$ is
$$
\Sel_{\underline{J}}(T^\dagger/K_\infty) \colonequals \ker \left( H^1_\Sigma(K_\infty,T^\dagger) \to \mathcal{P}_{\Sigma,\underline{J}}(T^\dagger/K_\infty) \right).
$$
\end{defn}

We define the fine Selmer group over $K_\infty$ as $$\Sel_p^0(T^\dagger/K_\infty):= \Ker\Big(H^1_\Sigma(K_\infty,T^\dagger) \rightarrow \prod_{v}H^1(K_{\infty,v},T^\dagger)\Big)$$
where $v$ runs through all the places of $K_\infty$. 
{\begin{remark}
In the literature, $\Sel_p^0(T^\dagger/K_\infty)$ is sometimes referred as the strict Selmer group instead.
\end{remark}}Let $\mathcal{X}_{\underline{J}}(T^\dagger /K_\infty)$ and $\mathcal{X}_0(T^\dagger /K_\infty)$ denote the Pontryagin dual of the $\underline{J}$-Selmer group and fine Selmer group respectively. {Recall that $\Sel_{\underline{J}}(T^\dagger/K_\infty)$ (resp. $\Sel_p^0(T^\dagger/K_\infty)$) is said to be $\Lambda$-cotorsion if $\mathcal{X}_{\underline{J}}(T^\dagger /K_\infty)$ (resp. $\mathcal{X}_0(T^\dagger /K_\infty)$) is a torsion module over $\Lambda$.}
\begin{remark}\label{changebasis}
\textbf{Change of basis.} The Selmer groups $\Sel_{\underline{J}}(T^\dagger/K_\infty)$  depend both on the indexing set $\underline{J}$ and the choice of the Hodge-compatible basis of the Dieudonné module $\mathbb{D}_{\mathrm{cris},\gq}(T)$  (defined after \textbf{(H.P)}). A change of Hodge-compatible basis will affect the Coleman maps $\col_{T,\gq,i}^\infty$
in the same way as described in \cite[Section 2.4]{buyukboduklei0}.
\end{remark}
\begin{remark}
Under the hypothesis $K(1) \cap K_\infty=K$, there is a unique prime above $\gq$ in $K_\infty$. In this case,
$$
H^1(K_\gq, T \otimes \Lambda^\iota ) \cong H^1_\text{Iw}(K_{\infty,v},T)
$$
where $v$ is the unique prime above $\gq$. Furthermore, $\Gamma=\Gamma_\gq$ and we get Coleman maps
$$
\col_{T,\gq,i}^{k_\infty}: H^1_{\text{Iw}}(K_{\infty,v},T) \to \Lambda.
$$
\end{remark}
\subsection{ Poitou--Tate exact sequence}

By \cite[Proposition A.3.2]{perrinriou95}, we have the following Poitou--Tate exact sequence
\begin{equation}\label{Poitou-Tate}
H^1_{\text{Iw}}(K_\infty,T) \to \frac{H^1(K_\gp,T\otimes \Lambda^\iota)}{\ker \col_{T,J_\gp}} \oplus \frac{H^1(K_{\gp^c},T\otimes \Lambda^\iota)}{\ker \col_{T,J_{\gp^c}}} \to \mathcal{X}_{\underline{J}}(T^\dagger /K_\infty) \to \mathcal{X}_0(T^\dagger /K_\infty)\to 0,
\end{equation}
where $H^i_{\text{Iw}}(K_\infty,T)$ is the inverse limit of $H^i(K_\Sigma/K_n,T)$ where $K \subseteq K_n \subseteq K_\infty$ {and the transition maps are the corestriction maps}.

%\JR{Generalize lemma 8.2 and proposition 8.3 of \cite{buyukbodukleiPLMS}. (We don't need these control theorems in our article but it is nice to have them).}

\section{Coleman maps control the Mordell--Weil ranks}\label{sec: mordell}

Let $A$ be an abelian variety over $K$ of dimension $g$ having supersingular reduction at all the primes in $K$ above $p$. Let $T=T_p(A)$ be the Tate module of $A$ and $V=T_p(A) \otimes \Qp$. Then $V=\mathcal{M}_p$ satisfies the hypotheses \textbf{(H.crys)}, \textbf{(H.HT)}, \textbf{(H.Frob)} and \textbf{(H.P)}. Furthermore, $T^\dagger = A^\vee [p^\infty]$ where $A^\vee$ denotes the dual abelian variety.
\subsection{Ranks of Iwasawa modules}

We say that $V$ satisfies weak Leopoldt's conjecture $\mathrm{Leop}(K_\infty,V)$ if the following equivalent statements are true.
\begin{proposition} The following statements are equivalent:
\begin{enumerate}
    \item $H^2_{\mathrm{Iw}}(K_\infty,T)$ is $\Lambda$-torsion,
    \item $H^2(K_\Sigma/K_\infty, T \otimes \Q_p/\Z_p)=0.$
\end{enumerate}
\end{proposition}

\begin{proof}
See \cite[Proposition A.5]{LZ14}.
\end{proof}
We recall the fine Selmer group over $K_\infty$ is given by $$\Sel_p^0(A^\vee/K_\infty)= \Ker\Big(H^1_\Sigma(K_\infty,A^\vee[p^\infty]) \rightarrow \prod_{v}H^1(K_{\infty,v},A^\vee[p^\infty])\Big).$$
Let $\mathcal{X}^0$ be the Pontryagin dual of this fine Selmer group over $K_\infty$. {Define
$$
\mathcal{I} \colonequals \{ \underline{J}=(J_\gp,J_{\gp^c}) : |J_\gp|+|J_{\gp^c}|=2g \},
$$}
where $J_\gq \subseteq \{1,...,2g\}$ for $\gq \in \{\gp,\gp^c\}$.
\begin{lemma}\label{lem1}
Suppose that there exists a $\underline{J} \in \mathcal{I}$ such that $\Sel_{\underline{J}}(A^\vee/K_\infty)$ is $\Lambda$-cotorsion. Then $\HIw(K_\infty,T)$ is of rank $2g$ over $\Lambda$.
\end{lemma}
\begin{proof}
By the Poitou--Tate exact sequence \eqref{Poitou-Tate}, $\Sel_p^0(A^\vee/K_\infty)^\vee$ is a quotient of $\Sel_{\underline{J}}(A^\vee/K_\infty)^\vee$ and hence is $\Lambda$-torsion. By \cite[Proposition A.3.2]{perrinriou95}, we have the exact sequence 
$$ 0 \rightarrow \Sel_p^0(A^\vee/K_\infty)^\vee \rightarrow H^2_{\mathrm{Iw}}(K_\infty,T) \rightarrow \prod_{v \in \Sigma}H^2_{\mathrm{Iw}}(K_{\infty,v},T). $$
Now $H^2_{\mathrm{Iw}}(K_{\infty,v},T)$ is $\Lambda$-torsion for all $v $ (see \cite[Theorem A.2]{LZ14}). Therefore $H^2_{\mathrm{Iw}}(K_\infty,T)$ is $\Lambda$-torsion and hence $\mathrm{Leop}(K_\infty,V)$ is true. Now since $K$ is totally imaginary, \cite[Corollary A.8]{LZ14} shows that the rank over $\Lambda$ of $\HIw(K_\infty,T)$ is $2g$.
\end{proof}
{We now discuss a partial result toward the converse of lemma \ref{lem1}. Suppose that $\mathrm{Leop}(K_\infty,V)$ is true. Then $H^1_{\mathrm{Iw}}(K_\infty,T)$ is of rank $2g$ over $\Lambda$ by \cite[Theorem A.4]{LZ14}. Assume that the images of $\col_{T,\gq,i}^{k_\infty}$ ($1 \leq i \leq 2g, \gq \in \{\gp,\gp^c\}$) are nonzero. Thus, if $v$ and $w$ are the unique places above $\gp$ and $\gp^c$ in $K_\infty$ respectively,
\begin{equation}\label{Pontryagin dual}
\frac{H^1_{\mathrm{Iw}}(K_{\infty,v},T) \bigoplus H^1_{\mathrm{Iw}}(K_{\infty,w},T)}{\ker \col_{T,J_{\gp}}\bigoplus \ker \col_{T,J_{\gp^c}}}
\end{equation}
is of rank $2g$ since $|J_\gp|+|J_{\gp^c}|=2g$ and $H^1_{\mathrm{Iw}}(K_{\infty,v},T) \bigoplus H^1_{\mathrm{Iw}}(K_{\infty,w},T)$ has rank $4g$ by \cite[Theorem A.2]{LZ14}. By \cite[proof of proposition 6]{greenberg89}, $\frac{H^1(K_{\infty,w},A^\vee[p^\infty])}{H^1_{\mathrm{ur}}(K_{\infty,w},A^\vee[p^\infty])}$ is $\Lambda$-cotorsion for $w\in \Sigma^\prime$ such that $w \nmid p$. We deduce, by local Tate duality, that the $\Lambda$-module \eqref{Pontryagin dual} has the same rank as the Pontryagin dual of $\mathcal{P}_{\Sigma,\underline{J}}(A^\vee[p^\infty]/K_\infty)$. Thus $\mathcal{P}_{\Sigma,\underline{J}}(A^\vee[p^\infty]/K_\infty)$ has $\Lambda$-corank $2g$ and it follows that $\Sel_{\underline{J}}(A^\vee/K_\infty)$ is the kernel of a morphism between two modules of the same corank. Being such a kernel is a necessary (but not sufficient) condition for $\Sel_{\underline{J}}(A^\vee/K_\infty)$ to be $\Lambda$-cotorsion.}

\begin{remark}\label{rem:4.3}
{Suppose that $A=E$ is an elliptic curve. In that case, the Coleman maps defined in section \ref{Selmer groups} are the same as the $\sharp/\flat$-Coleman maps in \cite{LeiSprung2020}
$$
\col_{\gq}^\sharp, \col_{\gq}^\flat : H^1(K_\gq,T_p(E)\otimes \Lambda^\iota) \to \Lambda(\Gamma).
$$
For $\bullet,\circ \in \{\sharp,\flat\}$, it is then possible to construct signed $p$-adic $L$-functions $\mathfrak{L}_p^{\bullet,\circ}(E/K)$. If we suppose that $\mathfrak{L}_p^{\bullet,\circ}(E/K)$ is nonzero, then \cite[Theorem 3.7]{castella2018iwasawa} shows that the signed Selmer group $\Sel^{\bullet,\circ}(E/K_\infty)$ is $\Lambda$-cotorsion. Moreover, when $a_p=0$, the $\sharp/\flat$-Selmer groups generalise Kim's $\pm/\pm$-Selmer groups \cite{kimdoublysigned}. The cotorsion-ness of those Selmer groups is known unconditionally by \cite[Remark 8.5]{LeiPal19} in the case $\underline{J}=(1,1)$ and $\underline{J}=(2,2)$ corresponding to $+/+$ and $-/-$ Selmer groups.}

\end{remark}

In the next lemma, we give a condition over the cyclotomic extension which will ensure that $\HIw(K_\infty,T)$ is of rank $2g$ over $\Lambda$.
\begin{lemma}\label{lem2}
Suppose that there exists any $\underline{J} \in \mathcal{I}$ such that $\Sel_{\underline{J}}(A^\vee/K^\cyc)$ is $\Z_p[[\Gal(K^\cyc/K)]]$-cotorsion. Then $\HIw(K_\infty,T)$ is of rank $2g$ over $\Lambda$.
\end{lemma}
\begin{proof}
Since $K$ is totally imaginary, in view of \cite[Corollary A.8]{LZ14}, it is sufficient to show that $\mathrm{Leop}(K_\infty,V)$ is true. By the discussion before Example A.7 of \textit{(loc.cit)}, $\mathrm{Leop}(K^\cyc,V)$ implies $\mathrm{Leop}(K_\infty,V)$. But \cite[Lemma 2.4(2)]{LP20}  implies that $\Sel_{\underline{J}}(A^\vee/K^\cyc)$ being cotorsion as a module over $\Z_p[[\Gal(K^\cyc/K)]]$ is a sufficient condition for $\mathrm{Leop}(K^\cyc,V)$ to hold (see Remark \ref{rem:true?}).
\end{proof}
\begin{remark}\label{rem:true?}
 Note that   \cite[Lemma 2.4(2)]{LP20}  says that it needs to assume that $\Sel_{\underline{J}}(A^\vee/K^\cyc)$ is cotorsion  for all $\underline{J}\in \mathcal{I}$ in order to ensure $\mathrm{Leop}(K^\cyc,V)$ holds. But a careful analysis of their proof reveals that the cotorsionness of $\Sel_{\underline{J}}(A^\vee/K^\mathrm{cyc})$ for only one $\underline{J}$ is sufficient.
\end{remark}
\begin{remark}
\cite[Proposition 3.28]{buyukboduklei0} gives a sufficient condition for $\Sel_{\underline{J}}(A^\vee/K(\mu_{p^\infty}))_+$ to be cotorsion (the plus notation is defined before proposition 3.14 of \cite{buyukboduklei0}). 
\end{remark}
\begin{remark}
   If the Bloch-Kato Selmer group $\Sel_{\mathrm{BK}}(A^\vee/K)$ is finite, it is known that $\Sel_{\underline{J}}(A^\vee/K^\cyc)$ is $\Z_p[[\Gal(K^\cyc/K)]]$-cotorsion (cf. \cite[Corollary 2.10 and Remark 2.2]{ponsinet}). In this case, it is also known that  $\Sel_{\underline{J}}(A^\vee/K_\infty)$ is $\Lambda$-cotorsion (cf. \cite[Corollary 1.1]{RaySprung2023}).
\end{remark}
\subsection{Bound using logarithmic matrices} \label{sub:bound}
%For $n \geq 0$, we define
%\begin{align}
%\Yprime_n %&:=\Coker\Big(\HIw(K_\i%nfty,T)_{\Gamma_n} %\rightarrow \prod_{v %\mid p}\Hf(K_{n,v},T)\B%ig).
%\end{align}
Let $F$ be a $p$-adic local field. Consider the Bloch--Kato's dual exponential map $$\mathrm{exp}^*: H^1(F,T) \rightarrow F \otimes \Dcris(T)$$ with kernel $H^1_f(F,T)$. Note that the image of the dual Bloch--Kato exponential map $\mathrm{exp}^*$ lands inside $F \otimes \mathrm{Fil}^0\Dcris(T)$ and $\mathrm{Fil}^0\Dcris(T)$ has $\Z_p$-rank $g=g_+=g_-$ by (\textbf{H.P}). Denote $$\Hf(F,T)=H^1(F,T)/H^1_f(F,T)$$ and $$\Hf(K_{n,p},T)=\prod_{v\mid p}\Hf(K_{n,v},T).$$ {For $n \geq 0$, we define
\begin{align}
\Yprime_n &:=\Coker\Big(\HIw(K_\infty,T)_{\Gamma_n} \rightarrow \prod_{v \mid p}\Hf(K_{n,v},T)\Big).
\end{align}}
%\JR{But isn't (4.2) the same as (4.1)?}

For $i \in [1,2g]$, let us suppose that we fix a family of classes $c_1,...,c_{2g} \in \HIw(K_\infty,T)$ such that $\HIw(K_\infty,T)/\langle c_1, \cdots, c_{2g} \rangle$ is $\Lambda$-torsion. Their existence is guaranteed since $\HIw(K_\infty,T)$ is of rank $2g$ (under the hypothesis of either lemma \ref{lem1} or lemma \ref{lem2}).
Let $\loc_{p,n}(c_i)$ be the image of $c_i$ in $H^1_{/f}(K_{n,p},T)$.

We write $W=\mu_{p^\infty} \times \mu_{p^\infty}$ and $\Z_p[w]=\Z_p[w_1,w_2]$ for $(w_1,w_2)\in W.$ If $F \in \Lambda$, we can evaluate $F$ at $w$ in the following way. We write $F$ as a power series in $\gamma_\gp-1$ and $\gamma_{\gp^c}-1$, say $F_0(\gamma_{\gp}-1,\gamma_{\gp^c}-1) $.  We then define $$F(w)=F_0(w_1-1,w_2-1).$$ For a $\Lambda$-module $M$, we write $M_w=M \otimes \Z_p[w]$ for the $\Z_p[w]$-module induced by this evaluation map.

{Let $o(w_1)$ (resp. $o(w_2)$) be the least $r_1$ (resp. $r_2$) such that $w_1^{p^{r_1}}=1$ (resp. $w_2^{p^{r_2}}=1$). Then $\Z_p[w]$ is a free $\Z_p$-module of rank $\varphi(p^r)$ where $r=\max\{o(w_1),o(w_2) \}$. Let $n \geq 1$ be an integer. There are exactly $p^{2n}$ elements $w \in W$ such that $w^{p^n}=1$;  two such elements $w$ are called conjugate if there is an automorphism of the algebraic closure of $\Q_p$ taking one to the other \cite[Section 2, p. 238]{CM}. The number of conjugates of a given $w$ is the rank of $\Z_p[w]$.}

\begin{lemma}\label{primary}
Let $n \geq 1$ be an integer. We have 
$$\rank_{\Z_p}\Yprime_n-\rank_{\Z_p}\Yprime_{n-1} \leq \sum_w\rank_{\Z_p}\Big(\frac{H^1_{/f}(K_{n,p},T)}{\langle\loc_{p,n}(c_1), \cdots , \loc_{p,n}(c_{2g})\rangle}\Big)_w$$
where the sum runs over conjugacy classes of $w \in W$ such that $w^{p^n}=(1,1)$, but $w^{p^{n-1}}\neq (1,1)$. 
\end{lemma}
\begin{proof}
We will essentially follow the line of sketch given in \cite[Corollary 5.7]{LeiSprung2020} with some modifications. 
The Bloch--Kato dual exponential map implies  $$\prod_{v \mid p}\Hf(K_{n,v},T) \otimes_{\Z_p} \Q_p \cong \bigoplus_{v \mid p} K_{n,v}^{\oplus g}.$$ This gives 
$$\Big(\Hf(K_{n,p},T)\otimes_{\Z_p}\Q_p\Big)_{\Gamma_{n-1}} = \Hf(K_{n-1,p},T)\otimes_{\Z_p}\Q_p$$ which implies $$\Big(\Hf(K_{n,p},T)\otimes_{\Z_p}\Q_p\Big)_{w} = \Big(\Hf(K_{n-1,p},T)\otimes_{\Z_p}\Q_p\Big)_w $$ for $w^{p^{n-1}}=(1,1)$. Now suppose $M$ is a $\Lambda$-module. Then \cite[Lemma 2.7]{CM} implies that $$\rank_{\Z_p}M_{\Gamma_n} =\sum_w\rank_{\Z_p}M_w$$ where the direct sum runs over the conjugacy classes of $w \in W$ such that $w^{p^n}=(1,1)$. Applying this for $M=\Hf(K_{n,p},T)$, we deduce 
$$\rank_{\Z_p}\Hf(K_{n,p},T)=\rank_{\Z_p}\Hf(K_{n-1,p},T) + \sum_w\rank_{\Z_p}\Big(\Hf(K_{n,p},T)\Big)_w,$$
where the sum runs over conjugacy classes of $w \in W$ such that $w^{p^n}=(1,1)$, but $w^{p^{n-1}}\neq (1,1)$. We obtain a similar decomposition for $\HIw(K_\infty,T)_{\Gamma_n}$. Let us abbreviate $\Hf(K_{n,p},T)$ by $H^1_n $ and $\HIw(K_\infty,T)_{\Gamma_n}$ by $H^1_{\mathrm{Iw},n}$.
Functoriality of coinvariance and tensor products gives the following commutative diagram 
\begin{equation*}
	\begin{tikzcd}
	H^1_n\otimes \Q_p   & \cong & H^1_{n-1}\otimes \Q_p   &\oplus & \Big(\oplus_{w \in W} \big(H^1_{n}\big)_w \otimes \Q_p\Big)   \\
	 H^1_{\mathrm{Iw},n} \otimes \Q_p \arrow[u]  &\cong &  H^1_{\mathrm{Iw},n-1} \otimes \Q_p \arrow[u]  &\oplus & \Big(\oplus_{w \in W} \big(H^1_{\mathrm{Iw},n}\big)_w \otimes \Q_p\Big) \arrow[u] 
	\end{tikzcd}
	\end{equation*}
	where the sum on the right runs over conjugacy classes of $w \in W$ such that $w^{p^n}=(1,1)$, but $w^{p^{n-1}}\neq (1,1)$. This implies
	$$\rank_{\Z_p}\Yprime_n -\rank_{\Z_p}\Yprime_{n-1} =\sum_w \rank_{\Z_p}(\Yprime_n)_w.$$ The lemma follows from noting that $\Yprime_n$ is a quotient of $\frac{H^1_{/f}(K_{n,p},T)}{\langle\loc_{p,n}(c_1), \cdots , \loc_{p,n}(c_{2g})\rangle}.$
\end{proof}
Suppose $\theta$ is a character of $\Gamma$ sending $(\gamma_\gp, \gamma_{\gp^c})$ to $(w_1,w_2) \in W.$ Then evaluating an element of $\Z_p[[X,Y]]$ at $X=w_1-1$ and $Y=w_2-1$ is equivalent to evaluating an element of $\Lambda$ at $\theta$. Let $w=(w_1,w_2) \in W$ with  $w_1$ a primitive $p^r$-th root of unity and $w_2$ a primitive $p^s$-th root of unity, then the corresponding character of $\Gamma$ has conductor $\gp^{r+1}(\gp^c)^{s+1}.$
Suppose $\theta$ be a character of $\Gamma$ of conductor $\mathfrak{p}^{r+1}(\mathfrak{p}^c)^{s+1}$.  When restricted to $\Gamma_{\mathfrak{p}}$ (respectively $\Gamma_{\mathfrak{p}^c}$), the character $\theta$ gives a character of $\Gamma_p$ whose cyclotomic part is of conductor $p^{r+1}$ (respectively $p^{s+1}$)).

{ Recall that $\underline{J}=(J_\gp,J_{\gp^c})$ where $|J_\gp|+|J_{\gp^c}|=2g$. \newline Let $\mathbf{z}=z_1 \wedge z_2 \wedge \cdots \wedge z_{2g} \in \bigwedge^{2g}\Big(\prod_{v \mid p}H^1(K_v, T \otimes \Z_p[[\Gamma_p]]^\iota)\Big).$ Then we define 
$$\col_{T, \underline{J}}^{k_\infty}(\mathbf{z})=\det\Big(\col_{T,v,i}^{k_\infty}(z_j)\Big)_{v \in \{\gp,\gp^c\},1\leq j \leq 2g, i \in J_v.}$$
}

\begin{lemma}\label{wedgelemma}
Let $\mathbf{z}=z_1 \wedge z_2 \wedge \cdots \wedge z_{2g} \in \bigwedge^{2g}\Big(\prod_{v \mid p}H^1(K_v, T \otimes \Z_p[[\Gamma_p]]^\iota)\Big)$ and let $\theta$ be a character of $\Gamma$ of conductor $\mathfrak{p}^{r+1}(\mathfrak{p}^c)^{s+1}$. Put $n \colonequals \max\{r,s\}$. Write $e_\theta$ for the idempotent corresponding to $\theta$. Then $e_\theta\cdot \mathbf{z}$ lies in $\bigwedge^{2g}\Big(\prod_{v \mid p}H^1_f(K_{n,v}, T)\Big)$  if and only if $$\sum_{\underline{J} \in \mathcal{I}}H_{\underline{I}_0,\underline{J},r,s}\col_{T,\underline{J}}^{k_\infty}(\mathbf{z})(\theta)=0.$$
\end{lemma}
\begin{proof}
If $v$ is a place above $\mathfrak{p}$ or $\mathfrak{p}^c$ in $K_{n}$, we have the dual exponential maps $$\mathrm{exp}_v^*: H^1(K_{n,v},T) \rightarrow K_{n,v} \otimes \mathds{D}_{\cris,v}(T).$$ Define $\Exp_{\gp}^*=\prod_{v \mid \gp}\Exp_v^*$ and $\Exp_{\gp^c}^*=\prod_{v \mid \gp^c}\Exp_v^*$ where the products run over the places dividing $\gp$ and $\gp^c$ in $K_n$. Then $e_\theta\cdot \mathbf{z} $ lies in $\bigwedge^{2g}\Big(\prod_{v \mid p}H^1_f(K_{n,v}, T)\Big)$ if and only if $$e_\theta \cdot \bigwedge_{1 \leq j \leq 2g}\Big(\Exp^*_\gp(z_j) \times \Exp^*_{\gp^c}(z_j)\Big) =0.$$
(Here  $\Exp^*_\gp(z_j)$ is  the dual Bloch--Kato exponential map applied to  the appropriate \textit{projection of} $z_j$). Via the interpolation property of the Loeffler--Zerbes' big logarithm map ,  this is equivalent to 
\begin{equation}\label{PRL}
\bigwedge_{1 \leq j \leq 2g}\Big(\Phi^{-r-1}\mathcal{L}_{T,\gp}^{k_\infty}(z_j) \times \Phi^{-s-1}\mathcal{L}_{T,\gp^c}^{k_\infty}(z_j)\Big)(\theta)=0.
\end{equation}
Applying $\Phi^{-r-1}$ to $\mathcal{L}_{T,\gp}^{k_\infty}(z_j)$ we obtain
$$\Phi^{-r-1}\mathcal{L}_{T,\gp}^{k_\infty}(z_j)(\theta)=(v_1,\cdots ,v_{2g})C_{\gp,r} \cdots C_{\gp,1}(\theta)\begin{bmatrix}
           \col_{T,\gp,1}^{k_\infty}(z_j)^{\sigma^{-r-1}}(\theta) \\
           \vdots \\
           \col_{T,\gp,2g}^{k_\infty}(z_j)^{\sigma^{-r-1}}(\theta)
         \end{bmatrix}.$$
Similarly, applying  $\Phi^{-s-1}$ to $\mathcal{L}_{T,\gp^c}^{k_\infty}(z_j)$ we obtain
$$\Phi^{-s-1}\mathcal{L}_{T,\gp^c}^{k_\infty}(z_j)(\theta)=(v_1,\cdots ,v_{2g})C_{\gp^c,s} \cdots C_{\gp^c,1}(\theta)\begin{bmatrix}
           \col_{T,\gp^c,1}^{k_\infty}(z_j)^{\sigma^{-s-1}}(\theta) \\
           \vdots \\
           \col_{T,\gp^c,2g}^{k_\infty}(z_j)^{\sigma^{-s-1}}(\theta)
         \end{bmatrix}.$$
         
%\JR{1) Are the above two equations correct? Or we have to go modulo something as in \cite[Page 6]{LP20}?  Do we have both $\sigma$ and $\theta$? 2) Should the definition of $\col_{T, \underline{J}}^{k_\infty}(\mathbf{z})$ include $\sigma$?} \green{I think that we don't need the modulo since we evaluate at $\theta$. The $\sigma$ comes from the operator $\Phi$ so I believe the equations are correct.} \JR{Thanks for answering question 1. What is the answer of question 2?} \green{I think the definition should not include $\sigma$. Since $\sigma$ is a field automorphism, the vanishing of $\sum_{\underline{I},\underline{J} \in \mathcal{I}}H_{\underline{I},\underline{J},r,s}\col_{T,\underline{J}}^{k_\infty}(\mathbf{z})(\theta)$ is independent of whether or not we put the $\sigma$.}

Note  that $H_{\gp,r}(\theta)=C_{\gp,r} \cdots C_{\gp,1}(\theta)$ and $H_{\gp^c,s}(\theta)=C_{\gp^c,s} \cdots C_{\gp^c,1}(\theta)$. Therefore by taking wedge product we obtain that \eqref{PRL} is equivalent to the vanishing of 
$$\sum_{\underline{I},\underline{J} \in \mathcal{I}}H_{\underline{I},\underline{J},r,s}\col_{T,\underline{J}}^{k_\infty}(\mathbf{z})(\theta).$$
Now we claim that $H_{\underline{I},\underline{J},r,s}(\theta)$ is zero unless $I=\underline{I}_0$ and $\underline{J} \in \mathcal{I}$. 

If $i \in \{g+1,...,2g\}$, then the entire $i$-th row of $C_{\gp,r}$ (respectively $C_{\gp^c,s}$) is divisible by $\Phi_{p^r}(\gamma_\gp)$ (respectively $\Phi_{p^s}(\gamma_{\gp^c})$). That is, the lower half of $H_{\gp,r}$ and $H_{\gp^c,s}$ are always zero. Hence, in order for a $2g \times 2g$ minor to be nonzero we should take the upper halves of  $H_{\gp,r}$ and $H_{\gp^c,s}$. Hence the claim follows.
\end{proof}
\begin{lemma}\label{wedge2lemma}
Let $\theta$ be a character of $\Gamma$ of conductor $\gp^{r+1}(\gp^c)^{s+1}$ and let $w=(\theta(\gamma_p), \theta(\gamma_{\gp^c}))\in W$. Let $n \colonequals \max \{r,s\}$. If $$\sum_{\underline{J} \in \mathcal{I}}H_{\underline{I}_0,\underline{J},r,s}\col_{T,\underline{J}}^{k_\infty}(\mathbf{c})(\theta)\neq0 $$ then $$\rank_{\Z_p}\Big(\frac{H^1_{/f}(K_{n,p},T)}{\langle\loc_{p,n}(c_1), \cdots , \loc_{p,n}(c_{2g})\rangle}\Big)_w=0.$$ Let $\Q_p(\theta)$ be the field obtained by adjoining the image of $\theta$. When  $\rank_{\Z_p}\Big(\frac{H^1_{/f}(K_{n,p},T)}{\langle\loc_{p,n}(c_1), \cdots , \loc_{p,n}(c_{2g})\rangle}\Big)_w$ is nonzero, it is always bounded by $2g\dim_{\Q_p}\Q_p(\theta).$
\end{lemma}
\begin{proof}
For $\gq \in \{\gp,\gp^c\}$, $$\prod_{v \mid \gq}\Hf(K_{n,v},T)_w \otimes \Q_p \cong \Q_p(\theta)^{\oplus g} .$$ So $$\bigwedge^{2g}\Big(
\Hf(K_{n,p},T)\Big)_w\otimes \Q_p=
\bigwedge^{2g}\Big(\prod_{v \mid \gp}\Hf(K_{n,v},T)_w \times \prod_{v \mid \gp^c}\Hf(K_{n,v},T)_w\Big) \otimes \Q_p$$ is a one dimensional $\Q_p(\theta)$-vector space. By lemma \ref{wedgelemma}, 
$$\rank_{\Z_p}\Big(\frac{\bigwedge^{2g}H^1_{/f}(K_{n,p},T)}{\loc_{p,n}(c_1)\wedge \cdots \wedge \loc_{p,n}(c_{2g})}\Big)_w>0$$  if and only if $$\sum_{\underline{J} \in \mathcal{I}}H_{\underline{I}_0,\underline{J},r,s}\col_{T,\underline{J}}^{k_\infty}(\mathbf{c})(\theta)=0. $$ Therefore under the hypothesis of lemma \ref{wedge2lemma}, 
$$\rank_{\Z_p}\Big(\frac{\bigwedge^{2g}H^1_{/f}(K_{n,p},T)}{\loc_{p,n}(c_1)\wedge \cdots \wedge \loc_{p,n}(c_{2g})}\Big)_w=0.$$ Note that $\bigwedge^{2g}\Hf(K_{n,p},T)_w \otimes \Q_p$ is a one dimensional $\Q_p(\theta)$-vector space and so the image of $\loc_{p,n}(c_1)\wedge \cdots \wedge \loc_{p,n}(c_{2g})$ in it is nonzero, since the rank of the quotient is zero. This implies that 
$\rank_{\Z_p}\Big(\frac{H^1_{/f}(K_{n,p},T)}{\langle\loc_{p,n}(c_1), \cdots , \loc_{p,n}(c_{2g})\rangle}\Big)_w=0.$
\end{proof}

We impose the following hypothesis whenever needed.

\textbf{(H-large)}: For $r,s$ such that $|r-s| \gg 0$, $\sum_{\underline{J} \in \mathcal{I}}H_{\underline{I}_0,\underline{J},r,s}\col_{T,\underline{J}}^{k_\infty}(\mathbf{c})(\theta)\neq0. $

%\textbf{(H-tor)}: For all $\underline{J} \in \mathcal{I}$, $\Sel_{\underline{J}}(A^{\vee}/K_\infty)^{\vee}$ is $\Lambda$-torsion.

\begin{lemma}
Suppose the hypothesis \textbf{(H-large)} holds. Let $n \geq 1$ and $\Upxi_n$ be the set of characters $\theta$ of $\Gamma$ which factor through $\Gamma_n$ but not $\Gamma_{n-1}$ such that $$\sum_{\underline{J} \in \mathcal{I}}H_{\underline{I}_0,\underline{J},r,s}\col_{T,\underline{J}}^{k_\infty}(\mathbf{c})(\theta)=0 $$ (with the conductor of $\theta$ being $\gp^{r+1}(\gp^c)^{s+1}$). Then the cardinality of $\Upxi_n$ is bounded independent of $n$.
\end{lemma}
\begin{proof}
If $\theta$ factors through $\Gamma_n$ but not $\Gamma_{n-1}$, then either $r$ or $s$ equals $n$. By the hypothesis \textbf{(H-large)}, there exists a fixed integer $n_0$ (independent of $n$) such that if $$\sum_{\underline{J} \in \mathcal{I}}H_{\underline{I}_0,\underline{J},r,s}\col_{T,\underline{J}}^{k_\infty}(\mathbf{c})(\theta)=0, $$ then $|r-s|\leq n_0$. If either $r$ or $s$ equals $n$, then the number of such $(r,s)$ is bounded and hence the result.
\end{proof}
\begin{proposition}\label{prop: bound1}
Suppose the hypothesis \textbf{(H-large)} holds. Then $\rank_{\Z_p}\Yprime_n = O(p^n)$.
\end{proposition}
\begin{proof}
Let $C_n$ be the cardinality of $\Upxi_n$. By lemmas \ref{wedge2lemma} and \ref{primary}, 
$$\rank_{\Z_p}\Yprime_n -\rank_{\Z_p}\Yprime_{n-1} \leq 2gC_n\varphi(p^n).$$ Under the hypothesis \textbf{(H-large)}, $C_n$ is less than a certain constant independent of $n$. Summing the inequality over $n$, we deduce that $$\rank_{\Z_p}\Yprime_n-\rank_{\Z_p}\Yprime_1 \leq \sum_{i=2}^nC\varphi(p^i)=O(p^n),$$
where $C$ is some constant depending on $g$ and independent of $n$.
\end{proof}

\begin{lemma}
Let $\underline{J} \in \mathcal{I}$ and suppose that $\Sel_{\underline{J}}(A^{\vee}/K_\infty)^{\vee}$ is $\Lambda$-torsion. Then $\col_{T,\underline{J}}^{k_\infty}(\mathbf{c})\neq 0$.
\end{lemma}
\begin{proof}
The fact that $\Sel_{\underline{J}}(A^{\vee}/K_\infty)^{\vee}$ is $\Lambda$-torsion combined with \eqref{Poitou-Tate} and lemma \ref{lem1} imply the desired result.
\end{proof}

\subsection{Block anti-diagonal matrices}\label{sec:blockanti}
Suppose we assume that $\Sel_{\underline{J}}(A^{\vee}/K_\infty)^{\vee}$ is $\Lambda$-torsion. Then it will imply that $\col_{T,\underline{J}}^{k_\infty}(\mathbf{c})(\theta)\neq 0$ for $|r-s|\gg 0$. But this is not enough to ensure that 
$\sum_{\underline{J} \in \mathcal{I}}H_{\underline{I}_0,\underline{J},r,s}\col_{T,\underline{J}}^{k_\infty}(\mathbf{c})(\theta)\neq0 $ since $H_{\underline{I}_0,\underline{J},r,s}(\theta)$ could be zero. In general, this is complicated to verify unless we are in some special cases where we can explicitly write down the matrices $C_{\varphi,v}$. This special case that we will be looking at can be thought of as an analogue for the case $a_p=0$ for supersingular elliptic curves as explained below. 

In this section, for each $v$, let us suppose that there exist a basis of $\mathbb{D}_{\cris,v}(T)$ such that the matrix $C_v$ is of the form  $\left[
\begin{array}{c|c}
0 & * \\
\hline
* & 0
\end{array}
\right],
$
where * is a $g \times g$ matrix defined over $\Z_p$. This means $\varphi(v_i)\notin \Fil^0\mathbb{D}_{\cris,v}(T)$ and $\varphi^2(v_i)\in \Fil^0\mathbb{D}_{\cris,v}(T)$  for all $i \in \{1,...,g\}$. Therefore, $C_{v,n}$ is of the form 
$$C_{v,n}= \left[
\begin{array}{c|c}
0 & B_{v,1} \\
\hline
\Phi_{p^n}(1+X)B_{v,2} & 0
\end{array}
\right],
$$
for some invertible $g \times g$ matrices $B_{v,1}$ and $B_{v,2}$ defined over $\Z_p$.  {In the context of elliptic curve with $a_p=0$, we may choose a basis $\{\omega\}$ for $\Fil^0 \mathbb{D}_{\cris,v}(T)$ and extend it to a basis $\{\omega,\varphi(\omega)\}$ of $\mathbb{D}_{\cris,v}(T)$. With this choice, the matrix $C_v$ is equal to the anti-diagonal matrix $\begin{bmatrix} 0 & -1 \\ 1 & 0 \end{bmatrix}$ \cite[Appendix 4]{buyukboduklei0}. In particular, the matrix $C_{v,n}$ will have the same structure as the matrix $C_n$ defined in \cite[p. 192]{LeiSprung2020}. In this case, it can be shown \cite[Proposition 5.12]{LeiSprung2020} that the hypothesis \textbf{(H-large)} is satisfied for such elliptic curves.}

For all $n \geq 1$, we fix a compatible system $\{\zeta_{p^n}: \zeta_{p^{n+1}}^p=\zeta_{p^n}\}$ of primitive $p^n$-th roots of unity and we write $\varepsilon_n=\zeta_{p^n}-1$. 

Recall that $\underline{I}_0=(I_{\gp,0}, I_{{\gp^c},0})$ and $I_{\gq,0} =  \{1,...,g\}$ for $\gq \in \{\gp,\gp^c\}.$ Let $\underline{I}_1$ be the complement of $\underline{I}_0$. That is 
$\underline{I}_1=(I_{\gp,1}, I_{{\gp^c},1})$ and $I_{\gq,1} =  \{g+1,...,2g\}$ for $\gq \in \{\gp,\gp^c\}.$  Let $\underline{I}_{\mathrm{mix}_{0,1}}=(I_{\gp,0},I_{\gp^c,1})$ and $\underline{I}_{\mathrm{mix}_{1,0}}=(I_{\gp,1},I_{\gp^c,0})$.

\begin{proposition}\label{block}
Suppose $C_\gp$ and $C_{\gp^c}$ are block anti-diagonal matrices and the Selmer groups $\Sel_{\underline{I}_0}(A^{{\vee}}/K_\infty)$,  $\Sel_{\underline{I}_1}(A^{{\vee}}/K_\infty)$, $\Sel_{\underline{I}_{\mathrm{mix}_{0,1}}}(A^{{\vee}}/K_\infty)$ and $\Sel_{\underline{I}_{\mathrm{mix}_{1,0}}}(A^{{\vee}}/K_\infty)$ are all $\Lambda$-cotorsion. Then the hypothesis \textbf{(H-large)} holds.
\end{proposition}
\begin{proof}
For any $k \geq 1$, let $\delta_k$ be the constant \begin{equation*}
  \delta_k =
    \begin{cases}
      \frac{\varepsilon_1}{\varepsilon_2}\cdot \frac{\varepsilon_3}{\varepsilon_4} \cdots \frac{\varepsilon_{k-2}}{\varepsilon_{k-1}} & \text{if $k$ is odd}\\
     \frac{\varepsilon_1}{\varepsilon_2}\cdot \frac{\varepsilon_3}{\varepsilon_4} \cdots \frac{\varepsilon_{k-1}}{\varepsilon_{k}} & \text{if $k$ is even.}
    \end{cases}       
\end{equation*}
Let $\theta$ be a character of $\Gamma$ of conductor $\gp^{r+1}(\gp^c)^{s+1}$. By \cite[Lemma 3.5]{LP20}, 
\begin{equation*}
  H_{\gq,k}(\zeta_{p^k}-1) =
    \begin{cases}
      \left[
\begin{array}{c|c}
0 & \delta_k(B_{\gq,1}B_{\gq,2})^{\frac{k-1}{2}}B_{\gq,1} \\
\hline
0 & 0
\end{array}
\right] & \text{if $k$ is odd,}\\
&\\
      \left[
\begin{array}{c|c}
\delta_k(B_{\gq,1}B_{\gq,2})^\frac{k}{2} & 0 \\
\hline
0 & 0
\end{array}
\right] & \text{if $k$ is even,}
    \end{cases}       
\end{equation*}
where $k=r$ if $\gq=\gp$ and $k=s$ if $\gq=\gp^c$.
Hence,
\begin{equation}\label{cases}
H_{\underline{I}_0,\underline{J},r,s}(\theta)=0 \text{ unless }
    \begin{cases}
      \underline{J}=\underline{I}_0 & \text{if $r$ is even and $s$ is even,}\\
      \underline{J}=\underline{I}_1 & \text{if $r$ is odd and $s$ is odd,}\\
      \underline{J}=(J_\gp,J_{\gp^c}), J_\gp=I_{\gp,1}, J_{\gp^c}=I_{\gp^c,0} & \text{if $r$ is odd and $s$ is even,}\\
      \underline{J}=(J_\gp,J_{\gp^c}), J_\gp=I_{\gp,0}, J_{\gp^c}=I_{\gp^c,1} & \text{if $r$ is even and $s$ is odd.}
    \end{cases}  
\end{equation}
Under the hypotheses on the Selmer groups being cotorsion $\Lambda$-modules, we obtain $\col_{T,\underline{J}}^{k_\infty}(\mathbf{c}) \neq 0$ for $\underline{J}=\underline{I}_0,  \underline{J}=\underline{I}_1, \underline{J}=\underline{I}_{\mathrm{mix}_{0,1}}$ and $ \underline{J}=\underline{I}_{\mathrm{mix}_{1,0}}$ depending on the parity of $r$ and $s$. Therefore, when $|r-s| \gg 0$, by \cite[Corollary 5.10]{LeiSprung2020}, there exists integers $a,b_1,b_2,c_1,c_2$ such that
\begin{equation}\label{tech}
\val_p\Big(\col_{T,\underline{J}}^{k_\infty}(\mathbf{c})(\theta)\Big)=a + \frac{b_{i}}{\varphi(p^r)} + \frac{c_{i}}{\varphi(p^s)}
\end{equation}
where $i=1$ if $r>s$ and $i=2$ if $s>r$. (One can see the proof of \cite[Proposition 5.9]{LeiSprung2020} on how to obtain these integers $a, b_1,b_2, c_1$ and $c_2$.) In particular, for $|r-s| \gg 0$, $\col_{T,\underline{J}}^{k_\infty}(\mathbf{c})(\theta) \neq 0$ and hence $$\sum_{\underline{J} \in \mathcal{I}}H_{\underline{I}_0,\underline{J},r,s}\col_{T,\underline{J}}^{k_\infty}(\mathbf{c})(\theta) = H_{\underline{I}_0,\underline{J},r,s}\col_{T,\underline{J}}^{k_\infty}(\mathbf{c})(\theta)\neq0,  $$
where $\underline{J}\in \{\underline{I}_0,  \underline{I}_1, \underline{I}_{\mathrm{mix}_{0,1}},\underline{I}_{\mathrm{mix}_{1,0}}\}$ depending on the parity of $r$ and $s$ given in \eqref{cases}. This shows that hypothesis \textbf{(H-large)} is satisfied.

\end{proof}

\begin{remark}
Proposition \ref{block} is a generalization of \cite[Lemma 3.6]{LP20} and it was already assumed (and conjectured) there that the Selmer groups are cotorsion \cite[Conjecture 2.2]{LP20}.
\end{remark}

%\begin{remark}
% Abelian varieties of $GL(2)$-type provide examples when $C_\gp$ and $C_{\gp^c}$ are block anti-diagonal matrices (see \cite[Section 3.3]{LP20}). 
%\end{remark}
{As mentioned previously, the Selmer groups $\Sel_{\underline{J}}(A^\vee [p^\infty]/K_\infty)$ not only depends on the choice of subset $\underline{J}$ but also on the choice of $\Zp$-bases for $\mathbb{D}_{\cris,\gp}(T)$ and  $\mathbb{D}_{\cris,\gp^c}(T)$. However, as we will see, the Selmer groups $\Sel_{\underline{I}_0}(A^{{\vee}}/K_\infty)$,  $\Sel_{\underline{I}_1}(A^{{\vee}}/K_\infty)$, $\Sel_{\underline{I}_{\mathrm{mix}_{0,1}}}(A^{{\vee}}/K_\infty)$ and $\Sel_{\underline{I}_{\mathrm{mix}_{1,0}}}(A^{{\vee}}/K_\infty)$ are canonically defined when $C_\gp$ and $C_{\gp^c}$ are block anti-diagonal matrices.
\begin{defn}\label{def:HC}
Fix a Hodge-compatible $\Zp$-basis $\mathcal{B}_\gq=\{v_1,\ldots,v_{2g}\}$ of $\mathbb{D}_{\cris,\gq}(T)$. Hypothesis \textbf{(H.HT)} implies that $\Fil^0\mathbb{D}_{\cris,\gq}(T)$ is a direct summand of $\mathbb{D}_{\cris,\gq}(T)$ \cite[Remark 2.3]{buyukboduklei0}. Let $N_\gq \subseteq \mathbb{D}_{\cris,\gq}(T)$ denote the free $\Zp$-module complementary to $\Fil^0\mathbb{D}_{\cris,\gq}(T)$ and generated by the set $\{v_{g+1},\ldots,v_{2g}\}$. We say that a $\Zp$-basis $\mathcal{B}_\gq^\prime=\{w_1,\ldots,w_{2g}\}$ is Hodge-compatible with the basis $\mathcal{B}_\gq$ if $\{w_1,\ldots,w_g\}$ (resp. $\{w_{g+1},\ldots,w_{2g}\}$) generated the submodule $\Fil^0\mathbb{D}_{\cris,\gq}(T)$ (resp. $N_\gq$).
\end{defn}
\begin{lemma}\label{block anti-diagonal}
Let $\mathcal{B}_\gq$ and $\mathcal{B}_\gq^\prime$ be two Hodge-compatible bases in the sense of definition \ref{Hodge-compatible}. Suppose that $C_\gq$ is a block anti-diagonal matrix with respect to the basis $\mathcal{B}_\gq$. Then $\mathcal{B}_\gq^\prime$ is Hodge-compatible with $\mathcal{B}_\gq$ in the sense of definition \ref{def:HC}.
\end{lemma}
\begin{proof}
Since $\Fil^0\mathbb{D}_{\cris,\gq}(T)$ is a direct summand of $\mathbb{D}_{\cris,\gq}(T)$, we get decompositions $\mathbb{D}_{\cris,\gq}(T)=\Fil^0\mathbb{D}_{\cris,\gq}(T)\bigoplus N_\gq$ and $\mathbb{D}_{\cris,\gq}(T)=\Fil^0\mathbb{D}_{\cris,\gq}(T)\bigoplus N_\gq^\prime$ where $N_\gq$ is generated by $\{v_{g+1},\ldots,v_{2g}\}$ and $N_\gq^\prime$ is generated by $\{w_{g+1},\ldots,w_{2g}\}$. Since $C_{\gq}$ is block anti-diagonal with respect to $\mathcal{B}_\gq$, we have that $\Span\{\varphi(v_1),\ldots,\varphi(v_g)\} \subseteq N_\gq$ and since $\varphi$ is injective we in fact have equality. By a similar argument, $\Span\{\varphi(v_1),\ldots,\varphi(v_g)\} = N_\gq^\prime$. Thus $\{w_{g+1},\ldots,w_{2g}\}$ generates $N_\gq$.
\end{proof}

Let $\underline{I}_{\text{BDP}_{2g,\emptyset}}=(\{1,...,2g\},\emptyset)$ and $\underline{I}_{\text{BDP}_{\emptyset,2g}}=(\emptyset, \{1,...,2g\})$. These indices should correspond to Selmer groups related with BDP type $p$-adic $L$-functions. 

\begin{proposition}\label{canonical}
Suppose that there exist Hodge-compatible bases $\mathcal{B}_\gp$ and $\mathcal{B}_{\gp^c}$ of $\mathbb{D}_{\cris,\gp}(T)$ and  $\mathbb{D}_{\cris,\gp^c}(T)$ such that $C_\gp$ and $C_{\gp^c}$ are block anti-diagonal with respect to these bases. Suppose that we have $\underline{J}\in \{\underline{I}_0,\underline{I}_1,\underline{I}_{\mathrm{mix}_{0,1}},\underline{I}_{\mathrm{mix}_{1,0}}, \underline{I}_{\text{BDP}_{2g,\emptyset}}, \underline{I}_{\text{BDP}_{\emptyset,2g}} \}$. Then $\Sel_{\underline{J}}(A^\vee [p^\infty]/K_\infty)$ does not depend on the choice of Hodge-compatible bases for $\mathbb{D}_{\cris,\gp}(T)$ and  $\mathbb{D}_{\cris,\gp^c}(T)$.
\end{proposition}}
In other words, once there exist Hodge-compatible bases yielding block anti-diagonal matrices, $\Sel_{\underline{J}}(A^\vee [p^\infty]/K_\infty)$ becomes independent of the choice of Hodge-compatible bases for $\underline{J}$ as in the statement of the proposition.
\begin{proof}
Write $\mathcal{B}_\gp=\{v_1,\ldots,v_{2g}\}$ and let $\mathcal{B}_\gp^\prime=\{w_1,\ldots,w_{2g}\}$ be another Hodge-compatible basis of $\mathbb{D}_{\cris,\gp}(T)$. Recall that Hodge-compatible means that $\{v_1,\ldots,v_g\}$ and $\{w_1,\ldots,w_g\}$ are bases for the $\Zp$-submodule $\Fil^0\mathbb{D}_{\cris,\gp}(T)$. Let $B_\gp$ be the change of basis matrix from $\mathcal{B}_\gp^\prime$ to $\mathcal{B}_\gp$. Then the hypothesis on both basis forces $B_\gp$ to be block diagonal. Write
$$
B_\gp = \left[ \begin{array}{c|c}
B_{1,1} & 0 \\
\hline
0 & B_{2,2}
\end{array} \right]
$$
where $B_{1,1},B_{2,2} \in \GL_g(\Zp)$. Let $\col_{T,\gp,v_i}^{k_\infty}$ be the $i$-th Coleman map as defined in section \ref{Selmer groups} with respect to the basis $\mathcal{B}_\gp$. Let $\col_{T,\gp,\mathcal{B}_\gp}^{k_\infty}$ denotes the vector of Coleman maps $( \col_{T,\gp,v_i}^{k_\infty} )_{i=1}^{2g}$ which we see as a column vector. Define similarly $\col_{T,\gp,\mathcal{B}_\gp^\prime}^{k_\infty}$ the column vector of Coleman maps defined with respect to $\mathcal{B}_\gp^\prime$. Since $C_\gp$ is block anti-diagonal, lemma \ref{block anti-diagonal} gives us that $\mathcal{B}_\gp^\prime$ is Hodge-compatible with $\mathcal{B}_\gp$. Hence we can use \cite[Lemma 2.16]{buyukboduklei0} to deduce that both vectors of Coleman maps are related by the linear transformation
$
\col_{T,\gp,\mathcal{B}_\gp^\prime}^{k_\infty} = B\cdot  \col_{T,\gp,\mathcal{B}_\gp}^{k_\infty}.
$
Thus,
$$
\begin{bmatrix} \col_{T,\gp,w_1}^{k_\infty} \\ \vdots \\ \col_{T,\gp,w_g}^{k_\infty} \end{bmatrix}=B_{1,1} \begin{bmatrix} \col_{T,\gp,v_1}^{k_\infty} \\ \vdots \\ \col_{T,\gp,v_g}^{k_\infty} \end{bmatrix}, \quad \begin{bmatrix} \col_{T,\gp,w_{g+1}}^{k_\infty} \\ \vdots \\ \col_{T,\gp,w_{2g}}^{k_\infty} \end{bmatrix}=B_{2,2} \begin{bmatrix} \col_{T,\gp,v_{g+1}}^{k_\infty} \\ \vdots \\ \col_{T,\gp,v_{2g}}^{k_\infty} \end{bmatrix}.
$$
This implies that $(\col_{T,\gp,w_i}^{k_\infty}(z))_{i=1}^g=0$ if and only if $(\col_{T,\gp,v_i}^{k_\infty}(z))_{i=1}^g=0$ since $B_{1,1}$ is invertible. The same goes for $(\col_{T,\gp,w_i}^{k_\infty}(z))_{i=g+1}^{2g}$ and $(\col_{T,\gp,v_i}^{k_\infty}(z))_{i=g+1}^{2g}$. We conclude that $\ker \col_{T,J_\gp}$ is independent of the choice of basis if $J_{\gp}$ is $I_{\gp,0}$, $I_{\gp,1}$ or $\{1,...,2g\}$. The same argument holds for $\gp$ replaced by $\gp^c$ and so the result follows.
\end{proof}
\begin{remark}
 It is easy to see that $C_\gp$ and $C_{\gp^c}$ remain block anti-diagonal matrices upon a change of basis. This follows because $B_\gq C_{\gq}B_{\gq}^{-1}$ is block anti-diagonal if $B_\gq$ is block diagonal and $C_\gq$ is block anti-diagonal for $\gq \in \{\gp,\gp^c\}$. 
\end{remark}
\section{The bound on the Mordell--Weil rank}\label{sec:MW}
Let $\Sigma$ be a finite set of primes of $K$ dividing $p$, the archimedean primes and the primes of bad reduction of $A^{\vee}$. Let $K_\Sigma$ be the maximal extension of $K$ unramified outside $\Sigma$. Let $H^1_\Sigma(K_n,T)$ be the Galois cohomology group $H^1(\Gal(K_\Sigma/K_n),T)$.

For $n \geq 0$, we define
\begin{align*}
\mathcal{Y}_n &:=\Coker\Big(H^1_\Sigma(K_n,T) \rightarrow \prod_{v \mid p}\Hf(K_{n,v},T)\Big),\\
\Sel_p^0(A^\vee/K_n)&:= \Ker\Big(H^1_\Sigma(K_n,A^\vee[p^\infty]) \rightarrow \prod_{v}H^1(K_{n,v},A^\vee[p^\infty])\Big).
\end{align*}

Let $\mathcal{X}_n^0$ and $\mathcal{X}_n$ be the dual of the fine Selmer group $\Sel_p^0(A^{\vee}/K_n)$ and the classical $p$-Selmer group $\Sel_p(A^{\vee}/K_n)$ respectively. The Poitou--Tate exact sequence \cite[Proposition A.3.2]{perrinriou95} gives 
$$H^1_\Sigma(K_n,T) \rightarrow \prod_{v \mid p}\frac{H^1(K_{n,v},T)}{H^1_f(K_{n,v},T)} \rightarrow \mathcal{X}_n \rightarrow \mathcal{X}_n^0  \rightarrow 0.$$
For $v \mid p$, as  $H^1_f(K_{n,v},T) \cong A(K_{n,v}) \otimes \Z_p$ (ref.  \cite[Remark 3.27]{buyukboduklei0}), we obtain the exact sequence, 
\begin{equation}\label{eq:exact}
0 \rightarrow \mathcal{Y}_n \rightarrow \mathcal{X}_n \rightarrow \mathcal{X}_n^0 \rightarrow 0.
\end{equation}

\begin{theorem}\label{bound coinvariant}
Let $M$ be a finitely generated $\Lambda$-module of rank $r$. Then,
$$
\rank_{\mathbb{Z}_p}M_{\Gamma_n} = rp^{2n}+ O(p^n).
$$
\end{theorem}
\begin{proof}
This is \cite[Theorem 1.10]{Ha00}.
\end{proof}

\begin{lemma}\label{torsion}
Let $w$ be a place above $\gq$ in $K_\infty$. We have that $A^\vee (K_{\infty,w})[p^\infty]=0$.
\end{lemma}

\begin{proof}
We follow the arguments in the proof of \cite[Lemma 1.1]{LP20}. Let $\kappa_\gq$ denote the residue field of $K_\gq$. By \cite[Lemma 5.11]{Mazur}, the reduction map $A^\vee (K_\gq) \to A^\vee (\kappa_\gq)$ induces an isomorphism on $p$-torsion points. Thus, the supersingularity of $A^\vee$ at $\gq$ implies that $A^\vee (K_\gq)[p^\infty]=0$. Furthermore, for a place $w$ above $\gq$ in $K_\infty$, the decomposition group $\Gal(K_{\infty,w}/K_\gq)$ is a pro-$p$ group. We deduce that $A^\vee (K_{\infty,w})$ also has no $p$-torsion by an application of the orbit-stabilizer theorem as shown by the following argument.

Let $v$ be a prime above $\gq$ in $K_n$. Suppose for the sake of contradiction that $A^\vee (K_{n,v})[p]$ is not trivial. Write
$$
A^\vee(K_{n,v})[p]= \bigcup_i \mathrm{Orb}(x_i)
$$
as a disjoint union of orbits under the action of $\Gal(K_{n,v}/K_\gq)$. If $x_i \neq 0$, $\Gal(K_{n,v}/K_\gq)$ cannot fix $x_i$ or else $x_i$ would be a $p$-torsion point in $A^\vee(K_\gq)$. So, if $x_i \neq 0$, then the orbit-stabilizer theorem
$$
|\mathrm{Orb}(x_i)| = \frac{|\Gal(K_{n,v}/K_\gq)|}{|\mathrm{Stab}(x_i)|}
$$
implies that $p$ divides $|\mathrm{Orb}(x_i)|$ since $\Gal(K_{n,v}/K_\gq)$ is a $p$-group. Thus, all the $|\mathrm{Orb}(x_i)|$ are divisible by $p$ except $\mathrm{Orb}(0)$ which is of cardinality $1$. Hence, $|A^\vee(K_{n,v})[p]| \equiv 1 \bmod p$. This is a contradiction, since  $A^\vee(K_{n,v})[p]$ must contain a subgroup of order $p$. We conclude that $A^\vee(K_{n,v})[p]=0$ and thus $A^\vee(K_{\infty,w})[p]=\bigcup_n A^\vee(K_{n,v})[p] =0$.
\end{proof}

\begin{remark}
If $w$ is a place above $\gq$ in $K(\mu_{p^\infty})$, then the fact that the torsion part of $A^\vee(K(\mu_{p^\infty})_w)$ is finite was first showed by Imai in \cite{Ima75}.
\end{remark}
\begin{proposition}\label{prop:fine}
Suppose $\Sel_{\underline{J}}(A^{\vee}/K_\infty)^{\vee}$ is $\Lambda$-torsion for some $\underline{J}\in \mathcal{I}$. Then, the $\Z_p$-rank of $\mathcal{X}_n^0$ is $O(p^n).$
\end{proposition}
\begin{proof}
We follow the outline of the proof of \cite[Proposition 5.5]{LeiSprung2020}. We begin by considering the commutative diagram with exact rows
\begin{equation*}
\begin{tikzcd}
0 \arrow{r} & \Sel_p^0(A^\vee /K_n) \arrow[r] \arrow[d,"\alpha_n"] & H^1(K_n,A^\vee [p^\infty]) \arrow{r} \arrow[d, "\beta_n"] & \prod_{v_n} H^1(K_{n,v_n},A^\vee [p^\infty]) \arrow[d,"\prod_{\gamma_n}"] \\
0 \arrow{r} & \Sel_p^0(A^\vee /K_\infty)^{\Gamma_n} \arrow{r} & H^1(K_\infty,A^\vee [p^\infty])^{\Gamma_n} \arrow{r} & \prod_{w}H^1(K_{\infty,w}, A^\vee [p^\infty])^{\Gamma_n},
\end{tikzcd}
\end{equation*}
where the vertical maps are restriction maps, the product in the first row runs over all the places $v_n$ of $K_n$ and the product in the second row over all the places $w$ of $K_\infty$. By lemma \ref{torsion}, $A^\vee (K_{\infty,w})[p^\infty]=0$ and hence $\prod \gamma_n$ is an isomorphism for all $v_n|p$ by the inflation-restriction exact sequence. As the $p$-torsion global points inject into the $p$-torsion local points, by the same argument, $\beta_n$ is also an isomorphism. If $v_n \nmid p$, \cite[page 270]{Gr03} shows that $\ker \gamma_n$ is finite and of bounded order as $K_n$ varies. The snake lemma then shows that $\ker \alpha_n=0$ and $\coker \alpha_n$ is finite with order bounded independently of $n$. By the proof of lemma \ref{lem1}, $\mathcal{X}^0$ is $\Lambda$-torsion. It follows from theorem \ref{bound coinvariant} and the control on $\alpha_n$ that 
$$
\rank_{\mathbb{Z}_p}\mathcal{X}_n^0=\rank_{\mathbb{Z}_p} (\mathcal{X}^0)_{\Gamma_n} = O(p^n).
$$
\end{proof}

\begin{theorem}\label{thm:mainone}
Suppose that $\Sel_{\underline{J}}(A^\vee/K_\infty)$ is $\Lambda$-cotorsion for some $\underline{J}\in \mathcal{I}$ and
hypothesis \textbf{(H-large)} hold. Then the bound for the Mordell--Weil rank is given by $\rank_{}A^\vee(K_n)=O(p^n).$
\end{theorem}
\begin{proof}
 Under hypothesis \textbf{(H-large)}, proposition \ref{prop: bound1} implies $\rank_{\Z_p}\Yprime_n=O(p^n)$. It is easy to see that there is a natural surjection $\Yprime_n \rightarrow \mathcal{Y}_n$. In particular $\rank_{Z_p} \mathcal{Y}_n \leq \rank_{\Z_p}\Yprime_n$. Then the theorem follows from the exact sequence \eqref{eq:exact} and proposition \ref{prop:fine}. 
\end{proof}

\begin{theorem}\label{thm:Main2}
Suppose $C_\gp$ and $C_{\gp^c}$ are block anti-diagonal matrices and the multi-signed Selmer groups $\Sel_{\underline{I}_0}(A^{{\vee}}/K_\infty)$,  $\Sel_{\underline{I}_1}(A^{{\vee}}/K_\infty)$, $\Sel_{\underline{I}_{\mathrm{mix}_{0,1}}}(A^{{\vee}}/K_\infty)$ and $\Sel_{\underline{I}_{\mathrm{mix}_{1,0}}}(A^{{\vee}}/K_\infty)$ are all cotorsion over $\Lambda$.  Then $\rank_{\Z_p}(\mathcal{X}_n) = O(p^n)$. Hence the bound for the Mordell--Weil rank is given by $\rank_{}A^{\vee}(K_n)=O(p^n).$
\end{theorem}
\begin{proof}
The hypotheses imply that the hypothesis \textbf{(H-large)} holds  by proposition \ref{block}. Rest of the proof follows from that of theorem \ref{thm:mainone}.
\end{proof}

\textbf{Examples.} Suppose that $A$ is an abelian variety defined over $\Q$ with good supersingular reduction at both prime over $p$ in $K$ such that the algebra of $\Q$-endomorphisms of $A$ contains a number field $E$ with $[E:\Q]=\dim A$. Such abelian varieties are said to be of $\GL_2$-type. Suppose further that the ring of integers $\mathcal{O}_E$ of $E$ is the ring of $\Q$-endomorphisms of $A$ and that $p$ is unramified in $E$. It follows that the $p$-adic Tate module of $A$ splits into
$$
T \cong \bigoplus_{\ell |p}T_\ell (A)
$$
where the direct sum runs over all the prime $\ell$ of $E$ above $p$ and $T_\ell (A)$ is a free $\mathcal{O}_\ell$-module of rank $2$ where $\mathcal{O}_\ell$ is the completion of $\mathcal{O}_E$ at $\ell$. Then, it was proved in \cite[Section 3.3]{LP20} that there exists a basis of $\mathbb{D}_{\cris,\gq}(T_\ell (A))$ where the action of $\varphi$ is given a matrix of the form $\begin{bmatrix} 0 & \frac{b_{\ell,\gq}}{p}\\ 1 & \frac{a_{\ell,\gq}}{p}\end{bmatrix}$ for some $a_{\ell,\gq} \in p\mathcal{O}_\ell$ and $b_{\ell,\gq} \in \mathcal{O}_\ell^\times$. If we assume that $a_{\ell,\gq}=0$ for all $\ell$ and $\gq \in \{\gp,\gp^c\}$, both matrices $C_\gq$ will be block anti-diagonal and theorem \ref{thm:Main2} holds.
\begin{remark}
Let $A$ be an abelian variety over a number field $F$ with good ordinary reduction at all the primes above $p$ in $F$. Let $F_\infty$ be a uniform admissible $p$-adic extension of $F$ of dimension $d \geq 2$. In this context, assume that the Selmer group over $F_\infty$ is cotorison over the corresponding Iwasawa algebra.  Under the stronger hypothesis of the  $\mathfrak{M}_H(G)$-conjecture, Hung and Lim  give an explicit bound on the growth of the Mordell--Weil rank of $A$ along $F_\infty$ (see \cite[Theorem 3.1]{HL20}). In the case of a supersingular elliptic curve, they also give an explicit bound for the Mordell--Weil rank of $E$ along the $\Zp^2$-extension of a quadratic imaginary field where $p$ splits (see \cite[Theorem 6.3, Conjecture 2]{HL20}). {Under the same assumptions, in the recent work \cite{ray2021asymptotic}, A. Ray proved stronger bounds for the Mordell--Weil rank of an elliptic curve with good ordinary reduction at all the primes above $p$ along the noncommutative extension $F(E[p^\infty])$ (see \cite[Theorem 2.5 and Remark 2.6]{ray2021asymptotic}.} Therefore, it is a natural question to generalize {those} result in the supersingular case when $F_\infty$ is a $\Z_p^d$-extension over $F$ containing the cyclotomic extension $F_\cyc$.  One might also ask if it is possible to remove this hypothesis on the $\mathfrak{M}_H(G)$-conjecture, or at least replace it with something weaker. 
{Another avenue to explore is the case of mixed reduction type. Let $E$ be an elliptic curve defined over a number field $F$ where $p$ splits completely. Suppose that $E$ has good reduction at primes above $p$. In \cite{leilimmordellweil}, Lei and Lim constructed multi-signed Selmer groups where they allow both ordinary or supersingular reduction at primes above $p$. Under the assumption that at least one prime is supersingular for $E$ and that the dual of the aforementioned Selmer groups are torsion over the appropriate Iwasawa algebra, they show \cite[Theorem 5.9]{leilimmordellweil} that the Mordell-Weil rank of $E$ stays bounded along the cyclotomic $\Zp$-extension of $F$. One may ask if this can be generalized to abelian varieties with mixed reduction type at primes above $p$ over more general extensions.} These are our future projects and needs further research.

%In the case of a supersingular elliptic curve, they also give an explicit bound for the rank of $E$ along the $\Zp^2$-extension of a quadratic imaginary field where $p$ splits. When $a_p=0$, the explicit bound is discussed 
\end{remark}

%\subsection{Block anti-diagonal matrices modulo $p$}
%\JR{There are some problems I am facing. Doesn't generalize very easily. Is there a  Weierstrass' preparation theorem in many variables? \cite{LP20} seems to use that in single variable. Also there is some problem in generalizing Lemma 3.8 of \cite{LP20}. Mixed parity of $r,s$ case is difficult. I need help. }

%\section{Conjectural $p$-adic $L$-functions and Coleman adapted Euler-Kolyvagin systems}
%\JR{Need to adapt arguments from \cite{buyukbodukleiPLMS} and \cite{buyukboduklei0} in the $\Z_p^2$-extension case. For elliptic curves the $p$-adic $L$-function for $\Z_p^2$-extension is in \cite{LZ14}. } 

	\section{Speculative remarks on multi-signed \texorpdfstring{$p$-adic $L$-functions for $\mathrm{GSp}(4)$}{p-adic L-functions for GSp(4)} by Chris Williams}\label{sec:speculative remarks}
	
	%This section is by
	%Chris Williams\footnote{CW would like to thank Antonio Lei and David Loeffler for very informative conversations on this topic, though any and all misconceptions are his own.}, who we thank for allowing us to include this in our paper.
	
	We thank Chris Williams\footnote{CW would like to thank Antonio Lei and David Loeffler for very informative conversations on this topic, though any and all misconceptions are his own.} for allowing us to include this in our paper.

	There are a number of works in the literature proving signed Iwasawa main conjectures relating signed Selmer groups to signed $p$-adic $L$-functions. As such, it is natural to ask if the Selmer groups of the present paper have corresponding $p$-adic $L$-functions. We finally make some (very) speculative remarks in this direction.

	\subsection{Elliptic curves}
	In the case of elliptic curves over $\Q$ with good supersingular reduction, there are two signed Selmer groups, and two `standard' $p$-adic $L$-functions, i.e.\ those interpolating $L$-values in the standard (and only) critical region. Via \cite{Pol03} these give rise to two signed $p$-adic $L$-functions, and the signed Iwasawa main conjectures \cite{castella2018iwasawa} relate all possible Selmer groups and all possible $p$-adic $L$-functions.
	
	Let $K$ be an imaginary quadratic field in which $p$ splits. Let $E$ be an elliptic curve over $K$ with good supersingular reduction at both primes above $p$.  Via modularity -- recently proved in many cases in \cite{CaraianiNewton} -- we expect that if it is not a $\mathbb{Q}$-curve, $E/K$ corresponds to a weight 2 cuspidal Bianchi newform $\mathcal{F}$ of level $\Gamma_0(\mathfrak{n})$, for some ideal $\mathfrak{n} \subset \mathcal{O}_K$ prime to $p\mathcal{O}_K$. It thus makes sense to discuss the theory on the automorphic side, for $\cF$ rather than $E$.
 
In this setting, the set $\mathcal{I}$ has size 6, containing the pairs
	\[
	(J_{\mathfrak{p}},J_{\mathfrak{p}^c}) = (\varnothing, \{1,2\}),\  (\{1\},\{1\}),\  (\{1\},\{2\}),\ (\{2\},\{1\}),\  (\{2\},\{2\}),\  (\{1,2\},\varnothing).
	\]
	These are the examples of Proposition \ref{canonical}, and thus all 6 correspond to canonical multi-signed Selmer groups in this case. On the analytic side, there are only \emph{4} multi-signed $p$-adic $L$-functions, constructed in \cite[\S5]{Loe14} from the 4 standard $p$-adic $L$-functions of \S4 \emph{op.\ cit}. These correspond to the 4 Selmer groups where $|J_{\mathfrak{p}}| = |J_{\mathfrak{p}^c}| = 1$. The discrepancy here arises as the standard $p$-adic $L$-functions do not account for all possible regions where the $L$-function has critical values. This phenomenon is discussed in \cite[Fig.\ 4.1]{bertolinidarmonprasanna13} or \cite[Fig.\ 6.1]{LLZ15}. There should be two additional two-variable $p$-adic $L$-functions of `BDP type', interpolating the critical $L$-values in the `non-standard' region. When $E$ can be defined over $\Q$, i.e.\ when $\cF$ is the base-change of a classical modular form, one- and two-variable $p$-adic $L$-functions interpolating values in this region were studied in \cite{bertolinidarmonprasanna13} and \cite{xinwanwanrankinselberg} respectively. The Selmer groups where $J_{\mathfrak{p}}$ or $J_{\mathfrak{p}^c}$ is empty should be related to these BDP type $p$-adic $L$-functions. Their construction in the non-base-change setting (that is, for `genuine' Bianchi modular forms) remains mysterious.
	
	We briefly elaborate why there are four standard $p$-adic $L$-functions. Recall the Bianchi modular form $\cF$ has level $\Gamma_0(\mathfrak{n})$. There are four $p$-refinements to level $\Gamma_0(p\mathfrak{n})$: there are independently two refinements at $\mathfrak{p}$ and two at $\mathfrak{p}^c$, corresponding to choices of roots of the Hecke polynomials at $\mathfrak{p}$ and $\mathfrak{p}^c$ (see \cite[\S2.1]{BW18}). This amounts to choosing one of the (four) Hecke eigenspaces in $S_2(\Gamma_0(p\mathfrak{n}))$ where the prime-to-$p$ Hecke operators act as they do on $\mathcal{F}$. Each  $p$-refinement yields a standard $p$-adic $L$-function.
	
	Rephrasing, $\mathcal{F}$ generates an automorphic representation $\pi$ of $\mathrm{GL}_2(\mathbb{A}_K)$ whose local components $\pi_{\mathfrak{p}}$ and $\pi_{\mathfrak{p}^c}$ are both unramified principal series. A $p$-refinement is a pair of (independent) choices of $U_{\mathfrak{p}}$ and $U_{\mathfrak{p}^c}$ eigenvalues in the 2-dimensional spaces $\pi_{\mathfrak{p}}^{\mathrm{Iw}_{\mathfrak{p}}}$ and $\pi_{\mathfrak{p}^c}^{\mathrm{Iw}_{\mathfrak{p}^c}}$ of Iwahori-invariant vectors.
	
	Summarising, in the case of elliptic curves, all of the 6 elements of the set $\mathcal{I}$ have attached multi-signed Selmer groups and $p$-adic $L$-functions, which we expect to be related by suitable Iwasawa main conjectures.

	\subsection{Abelian surfaces}
	
		Suppose now $2g=2$, i.e.\ $A$ is an abelian surface. Under the paramodularity conjecture a suitably `generic' abelian surface $A/K$ should correspond to a cuspidal automorphic representation $\Pi$ of $\mathrm{GSp}_4(\A_{K})$ (and this is known after base-change to some finite extension by \cite{BCGP-pot-mod}). If $A$ has good reduction at the primes above $p$, then $\Pi_{\mathfrak{p}}$ and $\Pi_{\mathfrak{p}^c}$ will be unramified. 
		
		In this case, the set $\mathcal{I}$ has size 70. By analogy to the elliptic curves case, one might naively expect that there are 70 multi-signed Selmer groups with 70 corresponding multi-signed $p$-adic $L$-functions. A key difference in this case, however, is that unlike weight 2 classical and Bianchi modular forms, where it exists $\Pi$ is \emph{not cohomological}. This causes some degeneration, which on the algebraic side is reflected in Proposition \ref{canonical} of the main text: only 6 elements of $\mathcal{I}$ lead to canonically defined multi-signed Coleman maps that are independent of the choice of Hodge-compatible basis. We now indicate how this degeneration occurs on the analytic side.

	\subsubsection{Counting $p$-adic $L$-functions: cohomological case}
	First consider the cohomological case, which is the closest direct analogue to the elliptic curves setting. Let $\Pi$ be a cohomological cuspidal automorphic representation of $\mathrm{GSp}_4(\A_K)$ unramified at the primes above $p$. As above, there should be standard $p$-adic $L$-functions attached to all $p$-refinements of $\Pi$ to sufficiently deep level at $p$. In this case, the \emph{Panchishkin condition} \cite{Pan94} predicts that we should look for Hecke eigensystems that arise in the Klingen-parahoric invariants of $\Pi_{\mathfrak{p}}$ and $\Pi_{\mathfrak{p}^c}$. One may compute that the Klingen-parahoric invariants of $\Pi_{\mathfrak{p}}$ are 4-dimensional (e.g.\ \cite[Prop.\ 3.15]{OST19}). Given $\Pi$ is cohomological, we expect that this 4-dimensional space should be the direct sum of four 1-dimensional Hecke eigenspaces, each with different Hecke eigenvalues. There should then be 16 standard $p$-adic $L$-functions for $\Pi$, corresponding to the 16 distinct pairs of choices of Klingen-invariant Hecke eigensystems in $\Pi_{\mathfrak{p}}$ and $\Pi_{\mathfrak{p}^c}$. 

	It is natural to expect  higher-weight analogues of the constructions of the present article, upon which we may ask which elements of $\mathcal{I}$ these $p$-adic $L$-functions should correspond to. An alternative version of this theory can be described after transferring $\Pi$ to an automorphic representation $\pi$ of $\mathrm{GL}_{4}(\A_K)$ via \cite{AS06}, as considered in \cite[\S3.3]{DJR18} or \cite[\S5--7]{BDGJW}. One sees that, for the Asgari--Shahidi conventions considered in \cite{BDGJW}, the 16 standard $p$-adic $L$-functions should correspond to Selmer groups where $J_{\mathfrak{p}}$ and $J_{\mathfrak{p}^c}$ are independently allowed to be one of $\{1,2\}, \{1,3\}, \{4,2\}, \{4,3\}$ (i.e. each can contain precisely one of $\{1,4\}$ and precisely one of $\{2,3\}$).  It is reasonable to expect that from these, one can construct 16 multi-signed $p$-adic $L$-functions (see \S\ref{sec:what is known}).	
	This accounts for only 16 of the 70 elements in $\mathcal{I}$:
	\begin{itemize}
		\item[(a)] There are a further 20 elements of $\mathcal{I}$ where each of $J_{\mathfrak{p}}$ and $J_{\mathfrak{p}^c}$ both have size 2, but are not of the special shape above: these should correspond to $p$-refinements of the $\mathrm{GL}_4$-representation $\pi$ that are not of `Shalika type'. This classification of the $p$-refinements is studied in detail in \cite{classical-locus}; in the language \emph{op.\ cit}., the 16 `good/Shalika-type' elements above exactly correspond to the refinements that are `$Q$-spin' at both primes above $p$, and the other 20 are not $Q$ spin at one or both primes above $p$.
		\item[(b)] There are 16 elements of $\mathcal{I}$ where $|J_{\mathfrak{p}}| = 1$ (so $|J_{\mathfrak{p}^c}| = 3$), and 16 when $|J_{\mathfrak{p}^c}| = 1$.
		\item[(c)] Finally, there are 2 cases where either $J_{\mathfrak{p}} = \varnothing$ or $J_{\mathfrak{p}^c} = \varnothing$. 
	\end{itemize}
All of this is reflected in the fact that $L(\Pi,s)$ now has many more critical regions, analogous for example to \cite[\S2.3]{LZ-GSp4-GL2}. The standard $p$-adic $L$-functions above and in (a) will only see the standard critical region. The 2 cases of type (c) should correspond to BDP-style $p$-adic $L$-functions, and will interpolate values in another of the critical regions. One might expect more types of $p$-adic $L$-functions, corresponding to type (b) elements of $\mathcal{I}$ and interpolating $L$-values in the other regions. This hints at a tantalising, but at present very mysterious, picture of Selmer groups and $p$-adic $L$-functions in this setting.

\subsubsection{Counting $p$-adic $L$-functions: non-cohomological case}
	Now suppose $\Pi$ is not cohomological, attached to an abelian variety $A$ with supersingular reduction at each prime above $p$. This causes degeneracy in the above picture.
	
	In the picture of \cite[\S2.3]{LZ-GSp4-GL2}, a number of the critical regions degenerate away to nothing: that is, some of the critical regions are empty in the non-cohomological case. (This is an analogue, for example, of the fact that the standard critical region is empty for weight 1 modular forms). In particular, in the type (b) cases above there is no hope of constructing $p$-adic $L$-functions with interpolative properties for non-cohomological $\Pi$. 
	
	There is also degeneracy in the 16 standard $p$-adic $L$-functions. In the supersingular case, the Frobenius eigenvalues at $p$ can only be $\pm \sqrt{p}$, each appearing twice. The 4-dimensional Klingen-invariants in $\Pi_{\mathfrak{p}}$ hence cannot split into four disjoint 1-dimensional Hecke eigenspaces, but only as a direct sum of two 2-dimensional Hecke eigenspaces, putting us in the so-called `irregular' setting. For $\mathrm{GL}_2$ over $\Q$ and $K$, the irregular setting was for example studied in \cite{BetWil20}. As a result, there should only be 4 standard $p$-adic $L$-functions in this case, corresponding to independent choices of eigenvalue $\pm \sqrt{p}$ at $\mathfrak{p}$ and $\mathfrak{p}^c$.
	
	These 4 standard $p$-adic $L$-functions should then have multi-signed analogues, corresponding to the 4 canonical Selmer groups in Proposition \ref{canonical} with $|J_{\mathfrak{p}}|= |J_{\mathfrak{p}^c}| = 2$.

	This picture holds in higher dimensions: for cohomological $\Pi$, the number of standard $p$-adic $L$-functions for $\mathrm{GSp}_{2g}$ will grow with $g$, but in the case of supersingular abelian varieties, there will only ever be 2 choices (again either $\pm\sqrt{p}$) of $p$-refinement for each of $\mathfrak{p}$ and $\mathfrak{p}^c$, reflected on the algebraic side in the four `standard' cases in Proposition \ref{canonical}.

The other cases of Proposition \ref{canonical}, where $J_{\mathfrak{p}}$ or $J_{\mathfrak{p}^c}$ is empty, provide good evidence that there exist two-variable BDP-style $p$-adic $L$-functions attached to abelian varieties.

	\subsubsection{What is known?}\label{sec:what is known}
	For cohomological automorphic representations of $\mathrm{GSp}_4$ over $K$, the existence of standard $p$-adic $L$-functions is now known for appropriate non-critical slope refinements (those of `Shalika type' above), but only after transfer to $\mathrm{GL}_4$ (see \cite{WilliamsGL2nNF}). The signed theory has not yet been explored in this setting. 
 
 We briefly also mention some related work when the base field is $\Q$, and where more is known. For cohomological representations $\Pi$ on $\mathrm{GSp}(4)/\Q$, standard $p$-adic $L$-functions for $\Pi$ have been constructed (after transfer to $\pi$ on $\mathrm{GL}(4)$) for Klingen-invariant eigensystems that correspond to non-ordinary, \emph{non-critical} $p$-refinements for $\mathrm{GL}(4)$ (see \cite{BDW20}). When $\pi$ is  `good supersingular' at $p$, and there are two non-critical slope $p$-refinements $\tilde\pi_1$ and $\tilde\pi_2$ of $\pi$, the recent papers \cite{Roc20} and \cite{LR-GL2n} construct two signed $p$-adic $L$-functions attached to $\pi$. They use the unbounded $p$-adic $L$-functions $L_p(\tilde\pi_1)$ and $L_p(\tilde\pi_2)$ of \cite{BDW20}, and prove that 
	\[
	L_p(\tilde\pi_1) \pm L_p(\tilde\pi_2) = L_p^\pm \times \mathrm{log}_\pi^\pm,
	\]
	where $L_p^\pm$ is a measure and $\mathrm{log}_\pi^\pm$ is Pollack's (unbounded) $\pm$-logarithmic distribution. The signed $p$-adic $L$-functions are then the measures $L_p^
	\pm$.
	
	One might expect the existence of two more (critical slope) signed $p$-adic $L$-functions in this case. This is presently not known, but one can guess at the shape of this theory. There are two further critical slope $p$-refinements of $\Pi$, transferring to critical slope $p$-refinements $\tilde\pi_3$ and $\tilde\pi_4$ of $\pi$ via the process in \cite[\S5]{BDGJW}. In the `supersingular' setting above, these refinements can be shown to have the same slope for the $U_p$ operator defined by $\mathrm{diag}(p1_n,1_n)$; hence, if $\tilde\pi_3$ and $\tilde\pi_4$ are both non-critical (cf.\ \cite[Rem.\ 3.15]{BDW20}), then the corresponding $p$-adic $L$-functions constructed in Theorem 6.23 \emph{op.\ cit}.\ have the same growth property. One might reasonably expect (in line with \cite{Roc20, LR-GL2n}) that $L_p(\tilde\pi_3) \pm L_p(\tilde\pi_4)$ is of the form [measure]$^\pm$ $\times$  [explicit unbounded logarithmic distribution]$^\pm$. However, unlike in \cite{Roc20,LR-GL2n}, it is not possible to deduce this solely using the interpolative properties of $L_p(\tilde\pi_3)$ and $L_p(\tilde\pi_4)$ alone: the growth is too big, so there are insufficient critical $L$-values to uniquely determine $L_p(\tilde\pi_3)\pm L_p(\tilde\pi_4).$
	
	The above constructions fundamentally use that $\Pi$ (hence $\pi$) appears in the Betti cohomology of locally symmetric spaces, and break down for non-cohomological $\Pi$. To construct $p$-adic $L$-functions for such $\Pi$, one strategy is to deform from cohomological to non-cohomological weight along an eigenvariety. In the present setting, this appears extremely difficult, as the points are non-regular, non-cohomological, and non-ordinary. In particular:
	\begin{itemize}
		\item[--] One must obtain good control over the $\mathrm{GSp}_4$-eigenvariety at non-regular, non-cohomological points, analogous to the study done in \cite{BerDimPoz} for $\mathrm{GL}_2/\Q$. 
		\item[--] In the $\mathrm{GSp}_4$ setting, the relevant points will not be ordinary, meaning one cannot study them using Hida families. Instead one must use Coleman families, which do not exist over the whole weight space but only over smaller neighbourhoods.
		\item[--] For $\mathrm{GSp}_4/K$, since $\mathrm{GSp}_4(\C)$ does not admit discrete series, the picture is even worse. In general, it is not expected that the $p$-refinements of $A/K$ will vary in classical families in the eigenvariety; this is a folklore conjecture \cite[Intro.]{Urb11}, a general analogue of \cite{CM09}. In this scenario deformation from cohomological weight is impossible.
	\end{itemize}
	
For $A/\Q$ ordinary, a $p$-adic $L$-function attached to its (unique) ordinary $p$-refinement is given in \cite[Prop.\ 3.3]{LZ-BSD}. However, the method of proof -- using higher Hida theory, via the coherent cohomology of Shimura varieties -- will not apply to the analogous setting over $K$, where the corresponding locally symmetric spaces are not even algebraic varieties. 
	
Summarising, for general supersingular $A$ defined over an imaginary quadratic field $K$ with $p$ split, constructing the 4 conjectured standard $p$-adic $L$-functions attached to $A/K$ seems completely out of reach with present methods.

Finally, consider the degenerate case where $A/K$ is a product of elliptic curves $E_1 \times \cdots \times E_g$, each $E_i$ defined over $K$. The Galois representation attached to $A$ will just be a direct sum of the Tate modules of the $E_i$, and its $L$-function will be a product of the $L(E_i,s)$. Good candidates for the multi-signed $p$-adic $L$-functions in this case would be products of the multi-signed $p$-adic $L$-functions attached to $E_i$ constructed in \cite[Prop.\ 5.2]{Loe14}, where we take the same choice in $\{++, +-, -+, --\}$ at each of the $E_i$.

\bibliographystyle{unsrt}
\bibliography{main}
\end{document}